\documentclass{siamart1116}

\usepackage{amsmath, amsfonts, enumerate, tikz, float, mathtools}

\newcommand{\Q}{\mathbb Q}
\newcommand{\N}{\mathbb{N}}
\newcommand{\Z}{\mathbb{Z}}
\newcommand{\R}{\mathbb{R}}
\newcommand{\cB}{\mathcal{B}}
\newcommand{\cL}{\mathcal{L}}

\newcommand{\cC}{\mathcal{C}}
\newcommand{\eps}{\varepsilon}
\newcommand{\dotvar}{\,\cdot\,}
\newcommand{\farrow}{\xrightarrow[]{d_f}}
\newcommand{\harrow}{\xrightarrow[]{d_H}}
\newcommand{\sprod}[2]{#1 \cdot #2 }
\newcommand{\setcond}[2]{\left\{ #1 \,:\, #2 \right\}}

\newcommand{\conv}{\operatorname{conv}}
\newcommand{\intr}{\operatorname{int}}
\newcommand{\Flt}{\operatorname{Flt}}
\newcommand{\cl}{\operatorname{cl}}
\newcommand{\vertices}{\operatorname{vert}}

\newcommand{\aff}{\operatorname{aff}}
\newcommand{\rec}{\operatorname{rec}}
\newcommand{\relint}{\operatorname{relint}}

\newcommand{\cone}{\operatorname{cone}}

\newcommand{\corn}{\operatorname{Cor}}
\newsiamremark{remark}{Remark}

{\bfseries}{\itshape}
\newtheorem{QUESTION}{Question}{\bfseries}{\itshape}

\newcommand{\TheTitle}{Approximation with intersection cuts} 
\newcommand{\TheAuthors}{Gennadiy Averkov, Amitabh Basu, and Joseph Paat}
\headers{\TheTitle}{\TheAuthors}

\title{Approximation of corner polyhedra with families of intersection cuts}
\author{Gennadiy Averkov \thanks{Institute of Mathematical Optimization, Faculty of Mathematics, University of Magdeburg, Germany.}
\and
Amitabh Basu \thanks{Dept. of Applied Mathematics and Statistics, The Johns Hopkins University.}
\and
Joseph Paat \thanks{Institute for Operations Research, ETH Z\"urich, Switzerland.}\funding{Amitabh Basu and Joseph Paat were supported in part by the NSF grant CMMI1452820.}}

\begin{document}
\maketitle
\begin{abstract}
We study the problem of approximating the corner polyhedron using intersection cuts derived from families of lattice-free sets in $\R^n$. In particular, we look at the problem of characterizing families that approximate the
corner polyhedron up to a constant factor, which depends only on $n$ and not the data or dimension of the corner polyhedron. The literature already
contains several results in this direction. In this paper, we use the maximum number of facets of lattice-free sets in a family as a
measure of its complexity and precisely characterize the level of complexity of a family required for constant factor approximations. As one of the main results, we show that, for each natural number $n$, a corner
polyhedron with $n$ basic integer variables and an arbitrary number of continuous non-basic variables is approximated  up to a constant factor by intersection cuts
from lattice-free sets with at most $i$ facets if $i> 2^{n-1}$ and that no such approximation is possible if $i \le 2^{n-1}$. When the approximation factor is allowed to depend on the denominator of the fractional vertex of the linear relaxation of the corner polyhedron, we show that the threshold is $i > n$ versus $i \leq n$. The tools introduced for proving such results are of independent interest for studying intersection cuts.\footnote{An abbreviated version of this work was accepted for presentation at the 19th International IPCO Conference at the University of Waterloo. Many of the proofs presented in this manuscript are omitted in the IPCO version, and results presented in this manuscript improve upon those in the IPCO version.} 
\end{abstract}



\section{Introduction}

Given $n,k\in \N$, a matrix $R:=(r_1,\ldots,r_k) \in \R^{n \times k}$ with columns $r_1,\ldots,r_k \in \R^n$, and a vector $f \in \R^n \setminus \Z^n$, the set 
\begin{equation*}
\textstyle	\corn(R,f) := \conv \setcond{s \in \R_{\ge 0}^k}{f + \sum_{i=1}^k s_i r_i \in \Z^n}
\end{equation*}
has been studied in the integer programming literature as a framework for deriving {\em cutting planes} for general mixed-integer programs. When both $R$ and $f$ are rational, the well-known Meyer's theorem (see~\cite{Meyer}) implies that $\corn(R,f)$ is a rational polyhedron. In the case of rational $(R,f)$, we will refer to $\corn(R,f)$ as the \emph{corner polyhedron for $(R,f)$}. The original definition of the corner polyhedron going back to~\cite{gom} involved the condition $s\in \Z^k_{\geq 0}$ rather than $s\in \R^k_{\geq 0}$. Since then, the term corner polyhedron has been used in a broader sense, with $s$ constrained to be a continuous, integer, or mixed-integer vector. See also Chapter 6 of~\cite{conforti2014integer} for a detailed discussion.

An inequality description of $\corn(R,f)$ can be obtained via gauge functions of lattice-free sets. We define $\cC^n$ to be the family of all $n$-dimensional, closed, convex subsets of $\R^n$.  A set $B\subseteq\R^n$ is \emph{lattice-free} if $B\in \cC^n$ and the interior of $B$ does not contain points of $\Z^n$. A lattice-free set is called \emph{maximal} if it is not a proper subset of another lattice-free set\footnote{Some sources do not impose the condition $\dim(B) = n$ in the definition of maximal lattice-free sets, but the case $\dim(B) < n$ is not needed in this paper.}. For $B\in \cC^n$ with $0\in \intr(B)$, the \emph{gauge function} $\psi_{B} : \R^n \rightarrow \R$ of $B$ is 
\[
\psi_{B}(r) := \inf \setcond{\lambda > 0}{r \in \lambda B}.
\]

Given a set $B\in \cC^n$ with $f \in \intr(B)$, the \emph{cut for $(R,f)$ generated by $B$} (or the \emph{$B$-cut} of $(R,f)$) is
\[
\textstyle	C_B(R,f) := \setcond{s \in \R_{\ge 0}^k}{\sum_{i=1}^k s_i \psi_{B-f}(r_i) \ge 1}.
\]
In the degenerate case where $f \in \R^n \setminus \intr(B)$, we define $C_B(R,f) := \R_{\ge 0}^k$. If $B$ is lattice-free, we call $C_B(R,f)$ an \emph{intersection cut}. Given a family $\cB\subseteq \cC^n$, we call the set 
\[
\textstyle	C_{\cB}(R,f) := \bigcap_{B \in \cB} C_B(R,f)
\]
the \emph{$\cB$-closure} for $(R,f)$. If the family $\cB$ is empty, then we define $C_{\cB}(R,f) := \R^k_{\geq0}$. It should be noted that standard definitions of $B$-cuts require $B$ to be a lattice-free set, see for example~\cite[page 187]{conforti2014integer}. It will be convenient for us to work with the more general concept, which does not insist that B is lattice-free. This allows us to cover cuts valid for the generalization of the corner polyhedron, in which $\Z^n$ is replaced by a more general set $S \subseteq \R^n$, see for example~\cite{conforti2013cut}. An interesting case that would deserve an independent study is $S=\Z_{\ge 0}^n$. 


%
Cuts can be partially ordered by set inclusion. For sets $B_1, B_2 \in \cC^n$ satisfying the inclusion $B_1 \subseteq B_2$, the respective cuts are related by the inclusion $C_{B_1}(R,f) \supseteq C_{B_2}(R,f)$ for each $(R,f)$. It is known that every lattice-free set is a subset of a maximal lattice-free set~\cite{bccz2}. Hence cuts generated by maximal lattice-free sets are the strongest ones among all intersection cuts. It is also known that maximal lattice-free sets are polyhedra~\cite{lovasz}. Therefore it is natural to focus on cuts generated by lattice-free polyhedra, since among these cuts are all of the strongest intersection cuts.


\begin{definition}[$\cL_i^n$, $i$-hedral closures, and $\cL_*^n$] For $i\in \N$, let $\cL_i^n$ denote the family of all lattice-free (not necessarily maximal) polyhedra in $\R^n$ with at most $i$ facets; we call $C_{\cL_i^n}(R,f)$ the $i$-hedral closure of $(R,f)$. Let $\cL_\ast^n$ denote the family of all lattice-free (not necessarily maximal) polyhedra in $\R^n$.  \end{definition}

\noindent Elements of $\cL_2^n$ are called \emph{lattice-free splits}, and the respective closure $C_{\cL_2^n}(R,f)$ is the well-known split closure of $(R,f)$, see~\cite[page 151]{conforti2014integer}. 

For every family of lattice-free sets $\cB$, the $\cB$-closure $C_{\cB}(R,f)$ is a relaxation of $\corn(R,f)$, which means that the inclusion $\corn(R,f)\subseteq C_{\cB}(R,f)$ holds for every choice of $(R,f)$. Furthermore, the equality $\corn(R,f) = C_{\cB}(R,f)$ is attained when $\cB$ contains all maximal lattice-free polyhedra and $(R,f)$ is rational~\cite{zambelli}. 
This implies that one approach to computing $\corn(R,f)$ for rational $(R,f)$ is to classify maximal lattice-free sets and compute cuts using the corresponding gauge functions. Recent work has focused on this classification \cite{DBLP:conf/ipco/AndersenLWW07,AKW-2015,averkov2011,dw-ipco}, and the classification was given for $n=2$ in~\cite{dw-ipco}. However, a classification is not known for any $n\geq 3$. Furthermore, even if such a classification was available for an arbitrary dimension $n$, the respective gauge functions could be difficult to compute, in general. In fact, the number $i$ of facets of an arbitrary maximal lattice-free polyhedron $B\subseteq \R^n$ can be as large as $2^n$, while the computation of the respective gauge function would require evaluation of $i$ scalar products in the generic case. 

In light of these difficulties, instead of fully describing $\corn(R,f)$ by classifying lattice-free sets, one can aim to find a small and simple family of lattice-free sets whose intersection cut closure approximates $\corn(R,f)$ within a desired tolerance~\cite{DBLP:journals/siamjo/AndersenWW09,ACGT2014,bbcm}. In other words, for a fixed $n\in \N$, one can search for a simple family $\cB$ of lattice-free sets and a constant $\alpha\geq1$ such that the inclusions
$$\corn(R,f)\subseteq C_{\cB}(R,f) \subseteq \frac{1}{\alpha}\corn(R,f)$$
hold for all $(R,f)$. The inclusion $\corn(R,f)\subseteq C_{\cB}(R,f)$ holds for all $(R,f)$. Thus, we ask the following:
\begin{QUESTION}\label{ques:one}
Given subfamilies $\cB$ and $\cL$ of $\cC^n$, does there exists an $\alpha \ge 1$ such that $C_\cB(R,f) \subseteq \frac{1}{\alpha} C_\cL(R,f)$ holds for all pairs $(R,f)$?
\end{QUESTION}
\begin{QUESTION}\label{ques:two}
Given subfamilies $\cB$ and $\cL$ of $\cC^n$, is it true that for every $f \in \Q^n \setminus \Z^n$ there exists an $\alpha \ge 1$ (possibly, depending on $f$) such that $C_\cB(R,f) \subseteq \frac{1}{\alpha} C_\cL(R,f)$ holds for every rational $R$?
\end{QUESTION}


In this paper, we focus on answering these questions. If such an $\alpha$ exists for either of the previous questions, then the $\cB$-closure approximates the $\cL$-closure within a factor of $\alpha$, i.e., the $\cB$-closure provides a finite approximation of the $\cL$-closure for all choices of $(R,f)$ (or for a fixed $f$ and all rational $R$). Since the corner polyhedron of $(R,f)$ coincides with $C_{\cL_*^n}(R,f)$ for rational $(R,f)$, we are particularly interested in studying the case $\cL = \cL_*^n$. On the other hand, as the number of facets is a natural measure for describing the complexity of maximal lattice-free sets, we are interested in the case $\cB = \cL_i^n$ with $i\in \N$. Other subfamilies $\cB$ and $\cL$ of $\cC^n$ (not necessarily subfamilies of lattice-free sets) may be of independent interest in future work. 

\medskip

\textbf{Notation and terminology.} For background on convex sets and polyhedra, see for example~\cite{barvinok},~\cite{rock}, and~\cite{rsconvex}, and for background on integer programming, see for example~\cite{conforti2014integer}.

We use $\N$ to denote the set of all positive integers. The value $n\in \N$ will always denote the dimension of the ambient space $\R^n$,
and the values $r_1, \dots, r_k$ will always denote the columns of $R$ with $k\in \N$ . 
Stating that a condition holds for every $R$ means that the condition holds for every $R\in \bigcup_{k=1}^\infty \R^{n\times k}$. Stating that a condition holds for every $(R,f)$ means the condition holds for every $R\in  \bigcup_{k=1}^\infty \R^{n\times k}$ and $f\in \R^n\setminus \Z^n$. 

Let $[m]:=\{1,\dots, m\}$ for $m\in \N$. For $X\subseteq \R^n$, we use $\aff(X)$, $\cone(X)$, $\conv(X)$, $\intr(X)$ to denote the affine hull of $X$, the convex conic hull of $X$, the convex hull of $X$, and the interior of $X$, respectively. For a closed, convex set $C$ in $\R^n$, we use $\rec(C)$ and $\relint(C)$ to denote the recession cone of $C$ and the relative interior of $C$, respectively. For a polyhedron $P\subseteq \R^n$, we use $\vertices(P)$ to denote the set of all vertices of $P$. For sets $A, B\subseteq \R^n$, the set $A+B := \{a+b: a\in A, b\in B\}$ is the \emph{Minkowski sum} of $A$ and $B$. We say the Minkowski sum of $A$ and $B$ is \emph{direct} if every point in $A+B$ has a unique representation as the sum $a+b$ with $a\in A$ and $b\in B$; in this case, we write $A\oplus B$ rather than $A+B$. For $i\in [n]$, $e_i$ denotes the $i$-th standard basis vector in $\R^n$. For vectors $u,v\in \R^n$, we use the notation $u\cdot v$ to denote the standard scalar product of $u$ and $v$.

%
\section{Summary of results}\label{section:summary}

Let $B$ and $L$ be sets in $\cC^n$, and let $\cB$ and $\cL$ be subfamilies of $\cC^n$. For $\alpha \ge 1$, we call $\frac{1}{\alpha} C_B(R,f)$ the \emph{$\alpha$-relaxation} of the cut $C_B(R,f)$. 
Analogously, we call $\frac{1}{\alpha} C_{\cB}(R,f)$ the \emph{$\alpha$-relaxation} of the $\cB$-closure $C_{\cB}(R,f)$. Using $\alpha$-relaxations, the relative strength of cuts and closures can be quantified naturally as follows. For $f \in \R^n \setminus \Z^n$, we introduce the functional 
\begin{equation}
	\label{eq:set_cut_comparison}
\textstyle	\rho_f(B,L) := \inf \setcond{\alpha > 0}{ C_B(R,f) \subseteq \frac{1}{\alpha} C_L(R,f)  \ \forall R}.
\end{equation}

The value $\rho_f(B,L)$ quantifies up to what extent $C_B(R,f)$ can `replace' $C_L(R,f)$ when $R$ varies. For $\alpha \ge 1$, the inclusion $C_B(R,f) \subseteq \frac{1}{\alpha} C_L(R,f)$ says that the cut $C_B(R,f)$ is at least as strong as the $\alpha$-relaxation of the cut $C_L(R,f)$. For $\alpha < 1$, the previous inclusion says that not just $C_B(R,f)$ but also the $\frac{1}{\alpha}$-relaxation of the cut $C_B(R,f)$ is at least as strong as the cut $C_L(R,f)$. Thus, if $\rho_f(B,L) < 1$, the $B$-cut of $(R,f)$ is stronger than the $L$-cut of $(R,f)$ for every $R$, and the value $\rho_f(B,L)$ quantifies how much stronger it is. If $1 < \rho_f(B,L) < \infty$, then the $B$-cut of $(R,f)$ is not stronger than the $L$-cut of $(R,f)$ but stronger than the $\alpha$-relaxation of the $L$-cut for some $\alpha>0$ independent of $R$, where the value $\rho_f(B,L)$ quantifies  up to what extent the $L$-cut should be relaxed. If $\rho_f(B,L) = \infty$, then $C_B(R,f)$ cannot `replace' $C_L(R,f)$ because there is no $\alpha \ge 1$ independent of $R$ such that $C_B(R,f)$ is stronger than the $\alpha$-relaxation of $C_L(R,f)$. 

In addition to comparing the cuts coming from two sets $B$ and $L$, we want to compare the relative strength of the family $\cB$ to the single set $L$, and the relative strength of the two families $\cB$ and $\cL$. We consider these comparisons when $f$ is fixed or arbitrary. For the case of a fixed $f \in \R^n \setminus \Z^ n$, we introduce the functional
\begin{equation*}
\textstyle	\rho_f(\cB,L) := \inf \setcond{\alpha > 0}{C_\cB(R,f) \subseteq \frac{1}{\alpha} C_L(R,f)\ \forall R},
\end{equation*}
which compares $\cB$-closures to $L$-cuts for a fixed $f$. We also introduce the functional \begin{equation*}
\textstyle	\rho_f(\cB,\cL) : = \inf \setcond{\alpha > 0}{C_\cB(R,f) \subseteq \frac{1}{\alpha} C_\cL(R,f)\ \forall R},
\end{equation*}
which compares $\cB$-closures to $\cL$-closures for a fixed $f$.  

The analysis of $\rho_f(\cB,\cL)$ can be reduced to the analysis of $\rho_f(\cB,L)$ for $L \in \cL$, since one obviously has
\begin{equation}
\label{eq:fixed_f_fam_fam}
	\rho_f(\cB,\cL) = \sup \setcond{\rho_f(\cB,L)}{L \in \cL}.
\end{equation}

For the analysis in the case of varying $f$, we introduce two functionals:
\begin{align*}
	\rho(\cB,L) & := \sup \setcond{\rho_f(\cB,L)}{f \in \R^n \setminus \Z^n},
	\\ 
	\rho(\cB,\cL) & := \sup \setcond{\rho_f(\cB,\cL)}{f \in \R^n \setminus \Z^n}.
\end{align*}
Observe that
\begin{align}
	\rho(\cB,L) & = \sup\setcond{\rho_f(\cB,L)}{f \in \intr(L)}, 
	\label{eqFamilySetvaryingf}
	\\ \rho(\cB,\cL) & = \sup\setcond{\rho_f(\cB,L)}{f \in \intr(L), \ L \in \cL}
	\label{eqFamilyFamilyvaryingf}.
\end{align}

For the setting where $\cB$ and $\cL$ are families of lattice-free sets, the functional $\rho(\cB,\cL)$ was introduced in \cite[\S~1.2]{ACGT2014}, where the authors initiated a systematic study for the case of $n=2$. In the case that $(R,f)$ is rational, since $C_{\cL_\ast^n}(R,f)=\corn(R,f)$, the value $\rho(\cB,\cL_\ast^n)$ describes how well $C_\cB(R,f)$ approximates $\corn(R,f)$. 


We are interested in using the functionals $\rho(\cB, \cL)$ and $\rho_f(\cB, \cL)$ to analyze how the strength of $\cL_i^n$ changes as $i$ grows and to compare $\cL_i^n$ with $\cL_\ast^n$ for different choices of $i$. In the trivial case $i=1$, the family $\cL_i^n$ is empty. So, it suffices to consider the case $i \ge 2$. For $i \ge 2^n$, well-known results on lattice-free sets yield $\rho(\cL_i^n,\cL_\ast^n)=1$, which roughly means that, in the context of Questions~\ref{ques:one} and~\ref{ques:two}, $\cL_i^n$ is as good as $\cL_\ast^n$ whenever $i \ge 2^n$. Thus, we can restrict our attention to the families $\cL_i^n$ with $2 \le i \le 2^n$, where $\cL_{i}^n$ with $i=2^n$ can be viewed as a `replacement' for $\cL_\ast^n$. Our first main result examines strength in the context of the functional $\rho(\cB, \cL)$.

\begin{theorem}
	\label{thmRelativeStrength}
	Let $i\in \N$ be such that $2\leq i\leq 2^n$. If $i\leq 2^{n-1}$, then $\rho(\cL_i^n,\cL_{i+1}^n)=\infty$. If $i>2^{n-1}$, then $\rho(\cL_i^n, \cL_\ast^n) \le 4 \Flt(n)$, where $\Flt(n)$ is the so-called flatness constant.\footnote{The flatness constant $\Flt(n)$ will be introduced in Section~\ref{sec:lf_sets}.}
\end{theorem}

From Theorem~\ref{thmRelativeStrength}, we immediately see that 
	\begin{itemize}
		\item for all natural numbers $2 \le i < j$, the value $\rho(\cL_i^n,\cL_j^n)$ is infinite if $i \le 2^{n-1}$, and finite, otherwise. 
		\item for every natural number $i \ge 2$, the value $\rho(\cL_i,\cL_\ast^n)$ is infinite for $i \le 2^{n-1}$, and finite, otherwise. 
	\end{itemize}
	Moreover, since the flatness-constant $\Flt(n)$ is known to be polynomially bounded in $n$, the values of $\rho(\cL_i^n,\cL_j^n)$ and $\rho(\cL_i,\cL_\ast^n)$ are polynomially bounded in $n$ whenever they are finite.

Let $f\in \R^n\setminus \Z^n$. Our second main result examines strength in the context of the functional $\rho_f(\cB, \cL)$. Since $\rho_f(\cL_i^n, \cL_*^n)\leq \rho(\cL_i^n, \cL_*^n)$ by definition of $\rho(\cB, \cL)$, Theorem~\ref{thmRelativeStrength} immediately implies that $\rho_f(\cL_i^n, \cL_*^n)\leq4\Flt(n)$ for $i >2^{n-1}$. The following theorem shows that in the case when $f$ is rational, the finiteness $\rho_f(\cL_i^n, \cL_*^n)<\infty$ holds for every $i >n$ (that is, already starting from $i=n+1$). Given $u \in \Q^n \setminus \{0\}$, we define the \emph{denominator of $u$} to be the minimum $s \in \N$ such that $s u\in \Z^n$.

\begin{theorem}
\label{thmRelativeStrength_fixed_f}
Let $i\in \N$ be such that $2\leq i\leq 2^n$. Let $f\in \Q^n\setminus \Z^n$, and let $s$ be the denominator of $f$. If $i\leq n$, then $\rho_f(\cL_i^n,\cL_{i+1}^n ) = \infty$. If $i>n$, then $\rho_f(\cL_i^n, \cL_\ast^n) < \Flt(n)4^{n-1}s$.
\end{theorem}

In light of Theorems~\ref{thmRelativeStrength} and~\ref{thmRelativeStrength_fixed_f}, upper bounds on $\rho_f(\cL_i^n, \cL_*^n)$ necessarily depend on $f$ for $n<i\leq 2^{n-1}$. An important point to note is that Theorem~\ref{thmRelativeStrength_fixed_f} assumes rationality of $f$. Rationality on $f$ or $R$ is not required for the other results in this paper.  The finite approximation parts of Theorems~\ref{thmRelativeStrength} and~\ref{thmRelativeStrength_fixed_f} are proved in Section~\ref{sec:finite_approx}, and the inapproximability parts are proved in Section~\ref{sec:infinite_approx}. 

Our main tool used in proving Theorems~\ref{thmRelativeStrength} and~\ref{thmRelativeStrength_fixed_f} is Theorem~\ref{thmFinitenessFamilyFamily}, which is an approximation result about general $B$-cuts and $\cB$-closures. We set up some notation necessary to state this result. 

\begin{definition}[$\cB_f$, $\cC_f^n$, and $f$-closed family]\label{def:cC_f} Let $f\in \R^n\setminus \Z^n$, and let $\cB\subseteq \cC^n$. By $\cB_f$ we denote the subfamily of all sets in $\cB$ containing $f$ in the interior. In particular, for $\cB = \cC^n$ we get the family $\cC_f^n$ of all $n$-dimensional, closed, convex sets containing $f$ in the interior. We say that $\cB$ is an $f$-closed family if the family $\setcond{\psi_{B-f}}{B \in \cB_f}$ of gauge functions is closed with respect to pointwise convergence within the family $\{\psi_{C-f}:C \in \cC^n_f\}$.

\end{definition} 
The following theorem provides characterizations of the conditions $\rho_f(\cB,\cL)<\infty$ and $\rho(\cB,\cL)<\infty$ under the topological assumption that $\cB$ is $f$-closed. The reason for introducing such an assumption is the following. The closure $C_\cB(R,f)$ is determined in terms of inequalities defining $B$-cuts, but there are more inequalities valid for $C_\cB(R,f)$ that arise as a limit of sequences of the $B$-cuts for $B \in \cB$. We will see later that such `limit-case inequalities' \emph{must} be taken into consideration for characterizing $\rho_f(\cB,\cL)<\infty$ and $\rho(\cB, \cL)<\infty$. Since the coefficients of the cut-defining inequalities are expressed in terms of gauge functions, we need a topology on $\cC^n$ that corresponds to convergence of gauge functions. Later, we will introduce this topology using an appropriate metric, but for the purpose of formulating Theorem~\ref{thmFinitenessFamilyFamily}, it is enough to use the notion of an $f$-closed family. 

\begin{theorem}[Qualitative One-for-all Theorem for two families]
	\label{thmFinitenessFamilyFamily}
	Let $\cB\subseteq \cC^n$. Let $\cL$ be a subfamily of polyhedra in $\cC^n$, and suppose there exists some $m\in \N$ such that every set $L\in \cL$ has at most $m$ facets. 
	Then the following hold:
	\begin{enumerate}[(a)]
		\item Suppose $\cB$ is $f$-closed for a fixed $f \in \R^n \setminus \Z^n$. Then $\rho_f(\cB,\cL)<\infty$ if and only if there exists $\mu\in (0,1)$ such that for every $L \in \cL_f$, some $B \in \cB$ satisfies $B\supseteq \mu L + (1-\mu) f$. 
		\item Suppose $\cB$ is $f$-closed for all $f\in \R^n\setminus \Z^n$. Then $\rho(\cB,\cL) < \infty$ if and only if there exists $\mu\in (0,1)$ such that for every $f\in \R^n\setminus \Z^n$ and every $L \in \cL_f$, some $B \in \cB$ satisfies $B\supseteq\mu L + (1-\mu) f$. 
	\end{enumerate}
	
\end{theorem}

It is not hard to see that the $\cB$-closure approximates the $\cL$-closure if and only if the $\cB$-closure approximates the $L$-cut for each $L\in \cL$. The somewhat surprising message of Theorem~\ref{thmFinitenessFamilyFamily} is that in order for the $\cB$-closure to approximate an $L$-cut, it is necessary for the $B$-cut from a \textit{single} well-chosen $B\in \cB$ to approximate the $L$-cut. So with a view towards constant factor approximations of $L$-cuts, there is no synergy of all $B$-cuts for $B\in \cB$ that contributes to the approximation. We call results of this type `one-for-all' results. The term \emph{qualitative} in Theorem~\ref{thmFinitenessFamilyFamily} refers to the fact that the result characterizes when the functional values $\rho_f(\cB,\cL)$ and $\rho(\cB,\cL)$ are finite; this is in contrast to the \textit{Quantitative} One-for-all Theorem (see Theorem~\ref{thmOneForAllFamilies}), which gives concrete bounds on $\rho_f(\cB,\cL)$ and $\rho(\cB,\cL)$.
The one-for-all idea is useful when deriving both positive and negative results. In Theorem~\ref{thmFinitenessFamilyFamily}, which is proved in Section~\ref{sec:oneforall}, the above informal message is expressed rigorously in convenient geometric terms.


\section{Basic material}\label{sec:prelims}

In this section, we collect basic results that are used for proving our main results.  
\subsection{Lattice-free sets}\label{sec:lf_sets}


The following result from~\cite{bccz2} is proved using Zorn's lemma.

\begin{proposition} \label{prop:all_in_max}
	Every lattice-free set $B$ in $\R^n$ is a subset of a maximal lattice-free set in $\R^n$. 
\end{proposition} 

A characterization of maximal lattice-free sets was given by Lov\'asz~\cite{lovasz}; see also~\cite{MR3027668} and~\cite{bccz}. We say that a transformation $T:\R^n\to \R^n$ is \emph{unimodular} if $T$ is an affine transformation having the form $T(x) = Ux+v$ for a unimodular matrix $U\in \Z^{n\times n}$ and some $v\in \Z^n$. 

\begin{theorem} \label{thmLovasz}
	Let $B$ be a lattice-free set in $\R^n$. Then the following hold:
	\begin{enumerate}[(a)]
		\item $B$ is maximal lattice-free if and only if $B$ is a lattice-free polyhedron and the relative interior of each facet of $B$ contains a point of $\Z^n$.
		\item If $B$ is maximal lattice-free, then $B$ is a polyhedron with at most $2^n$ facets. 
		\item If $B$ is an unbounded maximal lattice-free set, then up to a unimodular transformation, $B$ coincides  with $B' \times \R^k$, where $1 \le k \le n-1$ and $B'$ is a bounded maximal lattice-free subset of $\R^{n-k}$. 
	\end{enumerate}
\end{theorem}

The bound $2^n$ in Theorem~\ref{thmLovasz}$(b)$ is tight, and moreover, there is a maximal lattice-free polyhedron with $i$ facets for every $i\in \{2,...,2^n\}$ . 

\begin{lemma}
\label{lemma:existence_lattice_free}
Let $i\in \N$ be such that $2\leq i\leq 2^n$. Then there exists a maximal lattice-free polyhedron in $\R^n$ with exactly $i$ facets.
\end{lemma}

\begin{proof}

We first claim that there exists a family of $i$ pairwise disjoint faces $F_1, \dots,$ $F_i$ of the unit cube $[0,1]^n$ that cover the set $\{0,1\}^n$ of all its vertices.
This can be derived using a constructive argument. Start with a list of two disjoint $(n-1)$-dimensional faces of $[0,1]^n$, say, $[0,1]^{n-1}\times\{0\}$ and $[0,1]^{n-1}\times\{1\}$. While the maintained list of faces is of size strictly less than $i$, pick a face $F$ with $\dim(F)\geq 1$ and replace it with two disjoint faces of $F$ of dimension $\dim(F)-1$. In each iteration, the length of the list grows by one. Thus, after finitely many iterations, a list $F_1, \ldots, F_i$ of $i$ faces with the desired property is constructed. 

For each $j\in [i]$, let $u_j\in \R^n$ and $c_j\in \R$ be such that the inequality $u_j\cdot x\leq c_j$ is valid for $[0,1]^n$ and defines the face $F_j = \{x\in [0,1]^n: u_j\cdot x\leq c_j\}$. We claim that $B = \{x\in \R^n: u_j\cdot x\leq c_j,~j\in [i]\}$ is a lattice-free polyhedron with exactly $i$ facets. By Theorem~\ref{thmLovasz}$(a)$, $B$ is a maximal lattice-free set if $B$ is a lattice-free polyhedron and the relative interior of each facet of $B$ contains a point of $\Z^n$. By construction, $v\in \{0,1\}^n$ satisfies $u_j\cdot v = c_j$ if and only if $v\in F_j$. Since the faces $F_1, \dots, F_i$ were chosen to be pairwise disjoint, the vertices of $F_j$ are contained in the relative interior of the facet of $B$ corresponding to the inequality $u_j\cdot x\leq c_j$. Thus, it remains to show that $B$ is lattice-free. 


Let $z\in \Z^n$. If $z\in\{0,1\}^n$, then $z\not\in \intr(B)$ by construction of $B$. So assume $z\in \Z^n\setminus \{0,1\}^n$. Write $z$ as $z = (z_1, \dots, z_n)$ and let $v = (v_1, \dots, v_n)$ be the point of $\{0,1\}^n$ closest to $z$. That is, for every $t\in [n]$, one has $v_t\in \{0,1\}$ with $v_t=1$ if and only if $z_t\geq 1$. By construction, $p:=(1+\epsilon)v-\epsilon z$ belongs to $[0,1]^n$ if $\epsilon>0$ is sufficiently small. Hence $v$ is in the relative interior of a line segment joining $z$ and $p$. Thus, if $z$ was in $\intr(B)$, then $v$ would be too, which is a contradiction. 
\end{proof}

The finiteness of $\rho(\cL_i^n,\cL_{*}^n)<\infty$ for $i>2^{n-1}$ is shown by combining Theorem~\ref{thmFinitenessFamilyFamily} with the so-called flatness theorem. 
For every nonempty subset $X$ of $\R^n$, the \emph{width function} $w(X, \dotvar) : \R^n \rightarrow [0,\infty]$ of $X$ is defined to be
\[
w(X,u) := \sup_{x \in X} \sprod{x}{u} - \inf_{x \in X} \sprod{x}{u}.
\]
The value 	$w(X) := \inf_{u \in \Z^n \setminus \{0\}} w(X,u)$ is called the \emph{lattice width} of $X$.

\begin{theorem}[Flatness Theorem]\label{thm:flat}
	The value $$\Flt(n) := \sup \setcond{w(B)}{B~\text{is a lattice-free set in}~\R^n}$$
	is finite. 
\end{theorem}

The value $\Flt(n)$ is called the flatness constant in dimension $n$. Note that other sources define $\Flt(n)$ as the supremum of $w(K)$ among all compact sets $K\in \cC^n$ with $K\cap \Z^n = \emptyset$. Here we use an equivalent definition using maximal lattice-free sets; the equivalence can be derived using Proposition~\ref{prop:all_in_max} and Theorem~\ref{thmLovasz}$(c)$. The best currently known asymptotic upper bound on the flatness constant is $\Flt(n)\leq Cn^{3/2}$ for some absolute constant $C>0$~\cite{BLPS1999}. Unfortunately, the constant $C$ in the latter estimate is not known explicitly. The somewhat weaker bound of $\Flt(n)\leq n^{5/2}$ (see~\cite[page 317]{barvinok}) has the advantage of being explicit. 


\subsection{Gauge functions and the $f$-metric}\label{sec:gauge_f_metric}

Let $B\in \cC^n$ be such that $0\in\intr(B)$. The gauge function $\psi_B$ of $B$ 
is known to satisfy the following properties
\begin{align}
	\psi_B(r) &\ge0 & &\forall r \in \R^n,\label{eq:gauge_pos}
	\\ \psi_B(r_1+r_2)  &\le \psi_B(r_1) + \psi_B(r_2) & & \forall r_1, r_2 \in \R^n,\label{eq:gauge_subadditive}
	\\ \lambda \psi_B(r) &= \psi_B(\lambda r) &&  \forall r \in \R^n \ \forall \lambda \ge 0,\label{eq:gauge_pos_hom}
	\\ \lambda \psi_B(r) &= \psi_{\frac{1}{\lambda} B} (r) &&  \forall r \in \R^n \ \forall \lambda > 0,\label{eq:gauge_scale_hom}
\end{align}
and the equivalence
\begin{align}
\psi_B(r) & = 0 & & \Leftrightarrow&  r & \in \rec(B).\label{eq:gauge_rec}
\end{align}

Let $f\in \R^n\setminus \Z^n$. Here we introduce the $f$-metric on $\cC_f^n$ corresponding to the topological assumption on the family $\cB$ in Theorem~\ref{thmFinitenessFamilyFamily}. Our definition of the $f$-metric relies on the Hausdorff metric and polar bodies. The {\em Hausdorff metric $d_H$} is defined on the family of nonempty, compact subsets of $\R^n$ as follows: $d_H(A,B)$ is the minimum $\gamma\ge0$ such that $A \subseteq B + D(0, \gamma)$ and $B \subseteq A + D(0, \gamma)$, where $D(0, \gamma)$ is the closed ball of radius $\gamma$ around the origin. For a set $B\subseteq \R^n$, the {\em polar of $B$} is $$B^\circ:= \{r\in \R^n: r\cdot x\leq 1~\forall x\in B\}.$$

We define the \emph{$f$-metric} $d_f$ on $\cC_f^n$ to be
\begin{equation}
\label{metric_defn}
d_f(B_1, B_2) := d_H\left((B_1-f)^\circ, (B_2-f)^\circ\right).
\end{equation}
Since $f$ is in the interiors of $B_1$ and $B_2$, the sets $(B_1-f)^\circ$ and $(B_2-f)^\circ$ are compact, showing that $d_f(B_1,B_2)$ is well-defined. For a family $\cB\subseteq \cC_f^n$, we use the notation $\cl_f(\cB)$ to denote the closure of $\cB$ under the $f$-metric. 

Convergence in the $f$-metric can be expressed in several equivalent ways. In light of Proposition~\ref{prop:f_met_prop}, the $f$-closedness condition on $\cB$ that occurs in Theorem~\ref{thmFinitenessFamilyFamily} can be explained as the closedness of $\cB_f$ in the metric $d_f$. 

\begin{proposition}\label{prop:f_met_prop}
Let $f\in \R^n\setminus \Z^n$. Let $(B_t)_{t=1}^\infty$ be a sequence of sets in $\cC_f^n$ and $B$ a set in $\cC_f^n$. The following are equivalent:

\begin{enumerate}[(i)]
\item $\psi_{B_t-f}$ converges to $\psi_{B-f}$ pointwise, as $t\to\infty$,
\item $\psi_{B_t-f}$ converges to $\psi_{B-f}$ pointwise on the unit sphere, as $t\to\infty$,
\item $\psi_{B_t-f}$ converges to $\psi_{B-f}$ uniformly on the unit sphere, as $t\to\infty$,
\item $(B_t-f)^\circ\harrow (B-f)^\circ$,
\item $B_t \farrow B$.
\end{enumerate}
\end{proposition}
\begin{proof}
 The equivalence of $(i)$ and $(ii)$ follows from~\eqref{eq:gauge_pos_hom}. The equivalence of $(iii)$ and $(iv)$ follows from Theorem 1.8.11 in~\cite{rsconvex}, and the fact that the gauge function of a convex set containing the origin in its interior is equal to the support function of its polar (Theorem 1.7.6 in~\cite{rsconvex}). The equivalence of $(ii)$ and $(iii)$ follows from Theorem 1.8.12 in~\cite{rsconvex}. 
 The equivalence of $(iv)$ and $(v)$ follows by definition.
\end{proof}

The following result implies that lattice-free sets form a closed subset under the $f$-metric. 

\begin{proposition}	
	\label{prop:f_metric_properties}
	Let $f\in \R^n\setminus \Z^n$. Let $(B_t)_{t=1}^\infty$ be a sequence of sets in $\cC_f^n$ and $B$ a set in $\cC_f^n$ such that $B_t\farrow B$. Then the following hold:
	\begin{enumerate}[(a)]
	\item If $x\in B_t$ for each $t\in \N$, then $x\in B$.
	\item If $x\not\in \intr(B_t)$ for each $t\in \N$, then $x\not\in \intr(B)$. 
	\end{enumerate}

\end{proposition}

\begin{proof}
Note that $x\in B_t$ is equivalent to $\psi_{B_t-f}(x-f)\leq 1$. Furthermore, $x\in \intr(B_t)$ is equivalent to $\psi_{B_t-f}(x-f)< 1$. Thus, both $(a)$ and $(b)$ follow from Proposition~\ref{prop:f_met_prop}. 
\end{proof}

The next proposition shows that the topological assumptions of Theorem~\ref{thmFinitenessFamilyFamily} are fulfilled when applied to $\cB=\cL_i^n$. 

\begin{proposition}
	\label{prop:i_hedral_closed}
	Let $i\in \N$ and $f\in \R^n\setminus \Z^n$. Then $\cL_i^n$ is $f$-closed.
\end{proposition} 

\begin{proof}
Let $(B_t)_{t=1}^\infty$ be a sequence of sets in $\cL_i^n\cap \cC_f^n$ and $B$ a set in $\cC_f^n$ such that $B_t\farrow B$. For $t\in \N$, the polyhedron $B_t$ has at most $i$ facets, so $(B_t-f)^\circ$ has at most $i$ vertices. Thus, $(B_t-f)^\circ$ can be written as $(B_t-f)^\circ  =\conv(a_{1,t},.., a_{i,t})$, where $a_{1,t}, \ldots, a_{i,t} \in \R^n$ and some of these vectors may coincide. 
Since the sequence $((B_t-f)^\circ)_{t=1}^\infty$ is convergent in the Hausdorff metric, it is also bounded, and there is a bounded set $U$ in $\R^n$ such that $(B_t-f)^\circ\subseteq U$ for all $t\in \N$. Thus, for every $j\in [t]$, the sequence $(a_{j,t})_{t=1}^\infty$ is bounded. So, there exists a subsequence of $(a_{j,t})_{t=1}^\infty$ that converges to some $a_j\in \R^n$. Hence $\conv(a_1,\ldots,a_i)$ is the limit of $(B_t-f)^\circ$ in the Hausdorff metric. Since $(B_t-f)^\circ\harrow(B-f)^\circ$, the limit $\conv(a_1,\ldots,a_i)$ equals $(B-f)^\circ$. This shows that $B$ has at most $i$ facets. 
\end{proof}


\subsection{Basic properties of cuts and closures}\label{sec:strength}

The following proposition shows that $C_{\cB}(R,f)$ remains unchanged if $\cB$ is replaced by $\cl_f(\cB_f)$. 

\begin{proposition}
	\label{prop:closure_equality}
	Let $\cB\subseteq\cC^n$. Then $
		C_\cB(R,f) = C_{\cl_f(\cB_f)}(R,f)$
		 for every $(R,f)$. 
\end{proposition}

\begin{proof}
Recall that $B$-cuts are defined by linear inequalities whose coefficients involve the gauge function $\psi_{B-f}$. Since one can pass the coefficients to a limit in these inequalities, Proposition~\ref{prop:f_met_prop} implies that $
		C_\cB(R,f) = C_{\cl_f(\cB_f)}(R,f)$
		 for every $(R,f)$. 
\end{proof}

One might be interested in the relative strength of the family $\cL_i^n\setminus \cL_{i-1}^n$ of lattice-free polyhedra with exactly $i$-facets. It turns out that this family is not $f$-closed, so Theorem~\ref{thmFinitenessFamilyFamily} is not applicable. However, the $f$-closure of $\cL_i^n\setminus \cL_{i-1}^n$ is equal to $\cL_i^n\cap \cC_f^n$, giving further justification to study the relative strength of $\cL_i^n$. It should be pointed out that Proposition~\ref{prop:exactly_i_closed} is not used for proving our main results. 

\begin{proposition}\label{prop:exactly_i_closed}
Let $i\in \N$ and $f\in \R^n\setminus \Z^n$. Then $\cl_f((\cL_i^n\setminus \cL_{i-1}^n)\cap \cC_f^n)= \cL_i^n\cap \cC_f^n$.
\end{proposition}

\begin{proof}
We show the asserted equality of the two sets by verifying the inclusions `$\subseteq$' and `$\supseteq$'. The inclusion `$\subseteq$' follows from Proposition~\ref{prop:i_hedral_closed}. It is left to show the `$\supseteq$' inclusion. 
To verify `$\supseteq$', we consider an arbitrary $B \in \cL_i^n \cap \cC_f^n$ and $\epsilon>0$, and show the existence of $B' \in (\cL_i^n \setminus \cL_{i-1}^n) \cap \cC_f^n$ with $d_f(B,B') \le \epsilon$. Let $m$ be the number of facets of $B$. Then $(B -f)^\circ$ is a polytope with $m$ vertices. Appropriately adding $i-m$ points outside $B$ and at distance at most $\epsilon$ to $B$, we construct a polytope $P$ that contains $(B-f)^\circ$ as a subset and has exactly $i$ vertices. It follows that $B' := P^\circ + f$ is a polyhedron with exactly $i$ facets and is a subset of $B$. It is well known that $P^{\circ\circ} = P$. Thus, we have $(B' - f)^\circ = P$. This yields $d_f(B,B') = d_H( (B-f)^\circ, P) \le \epsilon$. 
\end{proof}

%
%
%

Propositions~\ref{propRelaxingCut} and~\ref{propStrongerRelCuts} are concerned with basic properties of $B$-cuts. Proposition~\ref{propRelaxingCut} follows from properties~\eqref{eq:gauge_pos}-\eqref{eq:gauge_rec} on gauge functions.

\begin{proposition} \label{propRelaxingCut}
	Let $f\in \R^n\setminus \Z^n$, $B\in \cC^n_f$, and $R\in \R^{n\times k}$ with $k\in \N$. Let $\alpha >0$. Then
	\[
		\frac{1}{\alpha} C_B(R,f) = C_{B'} (R,f),
	\]
	where $B'$ is a set in $\cC_f^n$ defined by
		$$B':= \frac{1}{\alpha} B + \left(1-\frac{1}{\alpha}\right) f.$$ 
\end{proposition}

\begin{proposition} \label{propStrongerRelCuts}
	Let $f\in \R^n\setminus \Z^n$ and $B, L\in \cC^n_f$. 
	Then the following conditions are equivalent: 
	\begin{enumerate}[(i)]
		\item $C_B(R,f) \subseteq C_L(R,f)$ holds for every $R$,
		\item $B \supseteq L$,
		\item $\psi_{B-f} \le \psi_{L-f}$.
		\item $\rho_f(B,L)\leq 1$. 
	\end{enumerate} 
\end{proposition}

\begin{proof} 

The equivalence $(i)\Leftrightarrow(iv)$ follows from the definition of $\rho_f(B,L)$, see~\eqref{eq:set_cut_comparison}. From the definition of gauge functions, we see that for a set $M\in \cC^n$ with $0\in \intr(M)$, the condition $r\in M$ is equivalent to $\psi_M(r)\leq 1$. This gives $(ii)\Leftrightarrow(iii)$. The implication $(iii)\Rightarrow(i)$ follows directly from the definition of $B$-cuts. To conclude the proof, it is sufficient to show that $(i)\Rightarrow (ii)$. If $(i)$ holds, then we get that the condition $1\in C_B((r),f)$ implies $1\in C_L((r),f)$ for every $r\in \R^n$. From the definition of $B$-cuts and gauge functions, we see that $1\in C_B((r),f)$ if and only if $r\not\in \intr(B)-f$. Thus, we obtain $\intr(B)\supseteq \intr(L)$ and therefore $B\supseteq L$. 
\end{proof}

Propositions~\ref{propRelaxingCut} and \ref{propStrongerRelCuts} give the following description of $\rho_f(B,L)$.

\begin{proposition}
	\label{propGeomRelativeStrength} 
	Let $f \in \R^n \setminus \Z^n$ and $B, L\in\cC^n$. Then one has
	\begin{align}
	\rho_f(B,L)  & = \inf \setcond{\alpha > 0}{B \supseteq \frac{1}{\alpha} L + \left(1- \frac{1}{\alpha}\right) f}  \nonumber
	\\ & = \inf \setcond{\alpha>0}{\alpha B + (1-\alpha) f \supseteq L} \nonumber 
	\\ & = \inf \setcond{\alpha > 0}{\psi_{B-f} \le \alpha \psi_{L-f}} & &  \label{eqRelStrengthCutsThroughGaugeFunc}
	\end{align}
	if $f \in \intr(L)$, and $\rho_f(B,L)  = 0$, otherwise. Moreover, if $f\in \intr(L)$ and the infimum in~\eqref{eqRelStrengthCutsThroughGaugeFunc} is finite, then all of the infima are attained for some $\alpha$. 
\end{proposition}

\begin{proof} If $f\not\in\intr(L)$, then $\rho_f(B,L)  = 0$ since $C_L(R,f) = \R^k_{\geq0}$ for all $R$. Thus, assume $f\in \intr(L)$. The desired result follows by applying Propositions~\ref{propRelaxingCut} and~\ref{propStrongerRelCuts}.

\end{proof}

\section
{One-for-all theorems}
\label{sec:oneforall}


For proving Theorem~\ref{thmFinitenessFamilyFamily}, we first derive an analogous result about approximation of a single set $L$ by a family $\cB$ in the case of a fixed $f$. In order to prove this analogous result, we apply the following lemma. We recall that by the Minkowski-Weyl theorem, every nonempty polyhedron can be represented as a Minkowski sum of a polytope and a finitely generated polyhedral cone (see Theorem 3.14 in~\cite{conforti2014integer}).

	\begin{lemma}\label{lem:approximating_L}
		Let $V$ be a nonempty finite subset of $\R^n$ and let $W$ be a nonempty finite subset of $\R^n \setminus \{0\}$. Let $L := \conv(V)+ \cone(W)$ and take $f \in \intr(L)$. Let $L_t := \conv(V \cup (f+ t W))$ for $t \in \N$. Then $f \in \intr(L_t)$ for every $t\in \N$, the sequence $(L_t)_{t=1}^\infty$ is increasing with respect to set inclusion, and $L_t \stackrel{d_f}{\to} L$, as $t \to \infty$.
	\end{lemma}
\begin{proof}

Assume to the contrary that $f\not\in \intr(L_t)$ for some $t\in \N$. Then $0 \not\in \intr(L_t - f)= \intr(\conv((V-f) \cup (tW))$. We apply a separation theorem for $0$ and $\intr(L_t - f)$ to deduce the existence of a vector $u \in \R^n \setminus \{0\}$ such that $\sprod{u}{x} \ge 0$ for all $x \in L_t -f$. This implies $\sprod{u}{x} \ge 0$ for all $x \in V-f$, giving $\sprod{u}{x} \ge 0$ for all $x \in \conv(V - f)$, as well as $\sprod{u}{x} \ge 0$ for all $x \in t W$, giving $\sprod{u}{x} \ge 0$ for all $x \in \cone(W)$. Combining these implications, we get that $\sprod{u}{x} \ge 0$ for all $x \in \conv(V-f)+ \cone(W) = L-f$, which gives $f \not\in\intr(L)$. This is a contradiction.

	Since $f\in L_t$, we can write $L_t$ as 
	\begin{equation*}
	L_t = \conv( V \cup \conv(\{f\}\cup(f + tW))).
	\end{equation*}
	The latter representation of $L_t$ shows that $L_t$ is increasing with respect to set inclusion. 
	
	Note that $L_t \stackrel{d_f}{\to} L$, as $t\to\infty$, if $(L_t-f)^\circ\harrow (L-f)^\circ$, as $t\to\infty$. According to Lemma 1.8.2 in~\cite{rsconvex}, we can prove the latter convergence by showing that $((L_t-f)^\circ)_{t=1}^\infty$ is a sequence of nonempty, compact, convex sets that is decreasing with respect to set inclusion and satisfies $\bigcap_{t=1}^\infty(L_t-f)^\circ =(L-f)^\circ$. For every $t\in \N$, we have $f\in \intr(L_t)$, so $((L_t-f))^\circ)_{t=1}^\infty$ is a sequence of nonempty, compact, convex sets. Since $L_t$ is increasing with respect to set inclusion as $t$ grows, the sequence $((L_t-f)^\circ)_{t=1}^\infty$ is decreasing with respect to set inclusion. Finally, we show that $\bigcap_{t=1}^\infty(L_t-f)^\circ =(L-f)^\circ$. Since $L_t \subseteq L$ for every $t\in \N$, it holds that $(L_t-f)^\circ \supseteq (L-f)^\circ$ for every $t$. This implies $\bigcap_{t=1}^\infty(L_t-f)^\circ \supseteq(L-f)^\circ$. Now consider an arbitrary vector $u\in \bigcap_{t=1}^\infty(L_t-f)^\circ$. So, $\sprod{u}{x} \le 1$ for every $t\in \N$ and $x \in L_t-f$. Recall that $L_t-f = \conv( (V-f) \cup t W)$. Thus, we get $\sprod{u}{x} \le 1$ for every $x \in V-f$ and $\sprod{u}{t x} \le 1$, or equivalently $\sprod{u}{ x} \le \frac{1}{t}$, for every $x \in W$ and every $t \in \N$. For an arbitrary $x \in W$, letting $t \to \infty$ in the inequality $\sprod{u}{ x} \le \frac{1}{t}$ gives us that $\sprod{u}{x} \le 0$. Taking the convex hull of $V-f$ and the convex conic hull of $W$, we arrive at $\sprod{u}{x} \le 1$ for all $x \in \conv(V-f)$ and $\sprod{u}{x} \le 0$ for all $x \in \cone(W)$. Thus, we get $\sprod{u}{x} \le 1$ for all $x \in \conv(V-f) + \cone(W) = L-f$, showing $\bigcap_{t=1}^\infty(L_t-f)^\circ \subseteq(L-f)^\circ$. 
\end{proof}

\begin{theorem}[One-for-all Theorem for a family and a set]\label{thmOneForAll_one_f} 
	Let $f\in \R^n\setminus \Z^n$. Let $\cB\subseteq \cC^n$ be $f$-closed, and let $L\in \cC^n$ be a polyhedron given by $L = \conv(V)+\cone(W)$, where $V$ is a nonempty finite subset of $\R^n$ and $W$ is a finite (possibly empty) subset of $\R^n\setminus \{0\}$. Then

\begin{equation}\label{eq:one_for_all_set_family}
		\frac{1}{|V|+|W|+1} \inf_{B \in \cB} \rho_f(B,L) \le \rho_f(\cB,L) \le \inf_{B \in \cB} \rho_f(B,L).
\end{equation}
		
\end{theorem}

\begin{proof}
	The validity of the asserted upper bound on $\rho_f(\cB, L)$ follows from the definition of $\rho_f(\cB, L)$ and $\rho_f(B,L)$. 
	
	For proving the asserted lower bound on $\rho_f(\cB,L)$, we first consider the case when $L$ is a polytope. It suffices to consider the case $\inf_{B \in \cB} \rho_f(B,L)> 0$ and $\rho_f(\cB,L) < \infty$. In this case, we have that $f\in \intr(L)$ and $\cB_f\neq\emptyset$. Let $\alpha>0$, and fix $\alpha':= \frac{ \alpha}{ k+1}$ for $k = |\vertices(L)|$. We will show that the condition
	\begin{equation}
	\label{eq:cond1}
	\forall R  \, : \, C_\cB(R,f) \subseteq \alpha C_L(R,f)
	\end{equation}
	implies the condition
	\begin{equation}
	\label{eq:cond2}
	\exists B \in \cB: \forall R,~ C_B(R,f) \subseteq \alpha' C_L(R,f).
	\end{equation}
		
Proposition~\ref{propGeomRelativeStrength} implies that $C_B(R,f) \subseteq \alpha'C_L(R,f)$ holds for every $R$ if and only if $B-f \supseteq \alpha'(L-f)$ holds. Thus, \eqref{eq:cond2} is equivalent to 
	\begin{equation}
	\exists B \in \cB \, : \, B -f \supseteq \alpha' (L-f),
	\end{equation}
	where $B - f \supseteq \alpha'(L-f)$ can also be reformulated as $B-f \supseteq \alpha' \vertices(L-f) $. Thus, we want to show
\begin{align*}
	&\forall R \,:\, C_\cB(R,f) \subseteq \alpha C_L(R,f) & &\Rightarrow & &\exists B \in \cB \,:\, B -f \supseteq \alpha' \vertices(L-f),
	\end{align*}
	or equivalently
	\begin{align}
	\label{contrapositive}
	&\forall B \in \cB : B-f \not\supseteq \alpha' \vertices(L-f)&&\Rightarrow&&\exists R \, : \, C_\cB(R,f) \not\subseteq \alpha C_L(R,f).
	\end{align}
	
	So assume that the premise of \eqref{contrapositive} is fulfilled. Let $r_1, \dots, r_k$ be the vertices of $L-f$ and set $R=(r_1,\ldots,r_k)$. For every $B \in \cB$, there exists $i \in [k]$ such that $\alpha' r_i \not \in B-f$ and $\psi_{B-f}(r_i) > 1/\alpha'$. This implies that $\alpha' \sum_{i=1}^k \psi_{B-f}(r_i) \ge 1$, and so $\alpha' (1,\ldots,1) \in C_\cB(R,f)$. Using
	$$\frac{\alpha'}{\alpha}\sum_{i=1}^k\psi_{L-f}(r_i) = \frac{k}{k+1}<1,$$
	we get that $\psi_{L-f}(r_i)=1$ for every $i \in [k]$. Hence $\alpha' (1,\ldots,1)\not\in \alpha C_{L}(R,f)$. Summarizing, condition \eqref{eq:cond1} implies condition \eqref{eq:cond2}. 

Consider any $\alpha>0$ such that $\rho_{f}(\cB,L) \leq \frac{1}{\alpha}$. By definition of $\rho_{f}(\cB,L)$, condition~\eqref{eq:cond1} holds for such an $\alpha$. Thus, condition~\eqref{eq:cond2} holds and there exists some $B\in \cB$ such that $C_B(R,f) \subseteq \alpha' C_L(R,f)$ holds for every $R$ with $\alpha' = \frac{\alpha}{k+1}$. Thus, we get $\inf_{\overline{B} \in \cB}\rho_f(\overline{B},L) \leq \rho_f(B,L)\leq \frac{(k+1)}{\alpha}$. Since $\alpha > 0$ was chosen arbitrarily with $\rho_{f}(\cB,L) \leq \frac{1}{\alpha}$, we obtain that $\inf_{\overline{B} \in \cB}\rho_f(\overline{B},L) \leq (k+1)\rho_{f}(\cB,L)$. This gives the desired lower bound in~\eqref{eq:one_for_all_set_family} as $k+1 = |\vertices(L)|+1\leq |V|+1 = |V|+|W|+1$.
	
	Now consider the case when $L$ is not a polytope. The proof idea is to approximate $L$ by a sequence of polytopes, use the proof of the polytope case and pass to the limit. For this, consider the lattice-free polytopes $L_t = \conv(V \cup (f + t W))$ with $t \in \N$. From Lemma~\ref{lem:approximating_L}, $f\in \intr(L_t)$ for every $t\in \N$,  the sequence $(L_t)_{t=1}^\infty$ is increasing with respect to set inclusion, and $L_t\farrow L$, as $t\to\infty$.
	
	
	Let $\alpha>0$ be such that $\rho_f(\cB, L)\leq 1/\alpha$. For every $R$, one has $C_\cB(R,f) \subseteq \alpha C_L(R,f)$. Fix $\alpha':= \alpha/ (|V|+|W|+1)$. Since $L_t$ is a polytope and $f\in \intr(L_t)$, we can apply the bounded case to obtain the implication 
	\begin{align}
	\label{eq:impl:t}
	&\forall R  \, : \, C_\cB(R,f) \subseteq \alpha C_{L_t}(R,f)  & &\Rightarrow & \exists B \in \cB \, : \, B - f \supseteq \alpha'(L_t-f).
	\end{align} 
	For every $t\in \N$, we  have $L_t \subseteq L$, which implies $C_L(R,f) \subseteq C_{L_t}(R,f)$ for all $R$. Thus, the premise of \eqref{eq:impl:t} holds for every $t\in \N$. It follows that for every $t\in \N$ there exists $B_t \in \cB$ with $B_t -f \supseteq \alpha'(L_t -f)$. Dualization of the latter containment gives $(B_t-f)^\circ \subseteq  (\alpha'(L_t-f))^\circ$ for every $t\in \N$.

	We claim there exists a subsequence of $((B_t-f)^\circ)_{t=1}^\infty$ that converges in the Hausdorff metric. There exists a sufficiently small closed ball $U$ that is contained in $\alpha'(L_1-f)$ and has a center at the origin. Since $(L_t)_{t=1}^\infty$ is increasing with respect to set inclusion, we have that $L_1\subseteq L_t$ for every $t\in \N$. Therefore $U\subseteq \alpha'(L_t-f)$ and $(\alpha'(L_t-f))^\circ\subseteq U^\circ$ for every $t\in \N$. It follows that $(B_t-f)^\circ\subseteq (\alpha'(L_t-f))^\circ\subseteq U^\circ$ for every $t\in \N$. Since $0\in \intr(U)$, the polar body $U^\circ$ is bounded, and thus the sequence $((B_t-f)^\circ)_{t=1}^\infty$ is bounded in the Hausdorff metric. By Blaschke's selection theorem~\cite[Theorem 1.8.6]{rsconvex}, there is a convergent subsequence, whose limit we represent as $(B-f)^\circ$. Since $\cB$ is $f$-closed, we get $B \in \cB$. Furthermore, since $(B_t-f)^\circ \subseteq  (\alpha'(L_t-f))^\circ$ for every $t\in \N$, when passing along such a convergent subsequence, we get the inclusion $(B-f)^\circ \subseteq (\alpha' (L-f))^\circ$. 
	
	Dualization of the previous containment gives $B-f \supseteq \alpha'  (L-f)$, which implies 
	$$\inf_{\overline{B}\in \cB}\rho_f(\overline{B},L)\leq \rho_f(B,L) \leq \frac{1}{\alpha'} = \frac{ |V|+|W|+1}{\alpha}.$$
	Dividing through by $|V|+|W|+1$, we get
	$$\frac{1}{|V|+|W|+1}\inf_{\overline{B}\in \cB}\rho_f(\overline{B},L)\leq \frac{1}{\alpha}.$$
	As the value $\alpha>0$ was arbitrarily chosen such that $\rho_f(\cB, L)\leq 1/\alpha$, the lower bound on $\rho_f(\cB, L)$ in~\eqref{eq:one_for_all_set_family} holds true.
	\end{proof}

An immediate corollary of Theorem~\ref{thmOneForAll_one_f} provides an analogous result in the case of an arbitrary $f$.

\begin{corollary}\label{cor:one_for_all_one_f}
	Let $\cB$ and $L$ be as in Theorem~\ref{thmOneForAll_one_f}. Further assume that $\cB$ is $f$-closed for all $f\in \R^n\setminus \Z^n$. Then
	\begin{equation}\label{eq:one_for_all_one_f-2}
	\frac{1}{|V|+|W|+1}{\adjustlimits\sup_{f\in \R^n\setminus \Z^n} \inf_{B \in \cB}} \rho_f(B,L) \le \rho(\cB,L) \le {\adjustlimits\sup_{f\in \R^n\setminus \Z^n} \inf_{B \in \cB} }\rho_f(B,L).
	\end{equation}
\end{corollary}

Theorem~\ref{thmOneForAll_one_f} and Corollary~\ref{cor:one_for_all_one_f} together give the following more general one-for-all type result, where $\cL$ is a family of sets. 

\begin{theorem}[Quantitative One-for-all Theorem for two families]
	\label{thmOneForAllFamilies}
	Let $\cB\subseteq \cC^n$. Let $\cL$ be a subfamily of polyhedra of $\cC^n$, for which there exists a constant $N\in \N$ satisfying the following property: every $L\in \cL$ has a representation $L = \conv(V)+\cone(W)$ using a nonempty finite subset $V$ of $\R^n$ and a finite (possibly empty) subset $W$ of $\R^n\setminus \{0\}$ for which $|V|+|W|+1\leq N$ holds. Then the following hold:
	\begin{enumerate}[(a)]
	\item Suppose $\cB$ is $f$-closed for a fixed $f\in \R^n\setminus \Z^n$. Then
	\begin{equation}~\label{eq:one_for_all_2_families_fixed_f}
		\frac{1}{N} {\adjustlimits \sup_{L \in \cL} \inf_{B \in \cB}} \rho_f(B,L) \le \rho_f(\cB,\cL) \le {\adjustlimits \sup_{L \in \cL} \inf_{B \in \cB}} \rho_f(B,L). 
	\end{equation}
	\item Suppose $\cB$ is $f$-closed for all $f\in \R^n\setminus \Z^n$. Then
	\begin{equation}\label{eq:one_for_all_2_families_fixed_f-2}
		\frac{1}{N} {\adjustlimits \sup_{L \in \cL, f \in \intr(L)} \inf_{B \in \cB}} \rho_f(B,L) \le \rho(\cB,\cL) \le {\adjustlimits \sup_{L \in \cL, f \in \intr(L)} \inf_{B \in \cB}} \rho_f(B,L).
	\end{equation} 
	\end{enumerate}
\end{theorem}

Note that~\eqref{eq:one_for_all_2_families_fixed_f} follows from~\eqref{eq:one_for_all_set_family}, and~\eqref{eq:one_for_all_2_families_fixed_f-2} follows from~\eqref{eq:one_for_all_one_f-2}. Using Theorem~\ref{thmOneForAllFamilies}, we are ready to prove Theorem~\ref{thmFinitenessFamilyFamily}.

\begin{proof}[Proof of Theorem~\ref{thmFinitenessFamilyFamily}]
By the Minkowski-Weyl theorem, each $L\in \cL$ can be written as $L = \conv(V)+\cone(W)$, where $V$ is a nonempty finite subset of $\R^n$, $W$ is a finite (possibly empty) subset of $\R^n\setminus\{0\}$, and $|V|+|W|+1\leq N(m,n)$ for a natural number $N(m,n)$. For sets $B,L\in \cC^n_f$ and $\mu\in (0,1)$, Proposition~\ref{propGeomRelativeStrength} implies that $\rho_f(B,L) \leq 1/\mu$ if and only if $B\supseteq \mu L+(1-\mu)f$. With this equivalence, Theorem~\ref{thmOneForAllFamilies} gives us the desired conclusion.
\end{proof}

\begin{remark}
As one can see from the proofs, the assertions of Theorem~\ref{thmOneForAll_one_f}  and Corollary~\ref{cor:one_for_all_one_f} are true without the $f$-closedness assumption on $\cB$ if the polyhedron $L$ is bounded. Similarly, the assertions of Theorems~\ref{thmFinitenessFamilyFamily} and~\ref{thmOneForAllFamilies} are true without the $f$-closedness assumption on $\cB$ if the family $\cL$ consists of bounded polyhedra only. \hfill $\diamond$
\end{remark}

\begin{remark}
If we want to use the $\cB$-closure for approximating other closures, we can replace $\cB$ by $\cl_f(\cB_f)$ and use Theorem~\ref{thmFinitenessFamilyFamily}. In general, passing to $\cl_f(\cB_f)$ is necessary. For example, let $L\subseteq \R^2$ be a lattice-free split, choose $f\in \intr(L)$, and let $\cB:= (\cL_3^2\setminus \cL_2^2)\cap\cC_f^n$ be the set of all lattice-free triangles containing $f$ in the interior. Since $L$ is a split, there is a nonzero vector $r$ in the lineality space of $L$. The intersection cut $C_L((r),f)$ is the empty set while $C_{B}((r),f)$ is nonempty for each $B\in \cB$. Hence $\rho_f(B, L) = \infty$ for each $B\in \cB$. However, it follows from Propositions~\ref{prop:closure_equality} and~\ref{prop:exactly_i_closed} that $\rho_f(\cB,L) \le 1$, see also~\cite[Theorem 1.4]{bbcm}. This example highlights the need for the topological assumption that $\cB$ is $f$-closed when characterizing the finite-approximation conditions $\rho_f(\cB, \cL)<\infty$ and $\rho(\cB, \cL)<\infty$. \hfill $\diamond$

\end{remark}

\section{Approximability results}\label{sec:finite_approx}

In this section we prove the finite approximation parts of Theorems~\ref{thmRelativeStrength} and~\ref{thmRelativeStrength_fixed_f} in Propositions~\ref{prop:finite_approx_f_arb} and~\ref{prop:proof_fixed_f_finite}, respectively. 

\subsection{Truncated cones}\label{sec:trunc_pyr}

This subsection presents a basic tool for proving finite approximation results.

\begin{definition}[Truncated cone]
Let $M\subseteq \R^n$ be a closed, convex set. Let $\alpha\geq 0$ and $p\in \R^n$ be such that for $M' := (1+\alpha) M+p$, the condition $\aff(M)\neq \aff(M')$ holds. Then $P:= \conv(M\cup M')$ is a truncated cone with bases $M$ and $M'$. 
\end{definition}

Note that in our definition of a truncated cone, the bases $M$ and $M'$ need not be bounded sets.

\begin{lemma}
	\label{lemma:tr_pyr}
	Let $p\in \R^n$, $\mu\geq 0$, and $M\subseteq \R^n$ be a closed convex set. Let $P$ be a truncated cone with bases $M$ and $M' = (1+\alpha)M+p$. Then the following hold:
	\begin{enumerate}[(a)]
		\item $P$ can be given by 
		\begin{equation}
			\label{eq:for:trunc:pyr}
			P = \bigcup_{ 0 \le \mu \le 1} \Bigl((1+\mu \alpha) M + \mu p \Bigr).
		\end{equation}

		\item Every point $f \in P$ can be given as 
		\[
			f= (1+ \mu \alpha) x + \mu p
		\] 
		for some $0 \le \mu \le 1$ and $x \in M$. Furthermore, if $\mu$ can be chosen such that $1/3 \le \mu \le 1$, then 
	\begin{equation}
		\label{eq:sandwiching}
		\frac{1}{4} P+\frac{3}{4} f \subseteq \conv( \{x\} \cup M') \subseteq P.
	\end{equation}
	\end{enumerate}
\end{lemma}

\begin{proof}

\begin{enumerate}[(a):]

	\item[$(a):$] Equality \eqref{eq:for:trunc:pyr} can be shown directly: 
	\[
		P = \conv(M \cup M') = \bigcup_{0 \le \mu \le 1} ((1-\mu) M + \mu M') = \bigcup_{ 0 \le \mu \le 1} \Bigl((1+\mu \alpha) M + \mu p \Bigr).
	\]
	
	\item[$(b):$] The representation for $f$ follows from $(a)$. The inclusion $\conv(\{x\} \cup M') \subseteq P$ is clear since $x \in M$. We show $\frac{1}{4} P +\frac{3}{4} f\subseteq \conv(\{x\} \cup M')$. Since $P=\conv(M \cup M')$, it suffices to check
	\begin{align*}
		\frac{1}{4} M+\frac{3}{4} f  & \subseteq \conv( \{x\} \cup M'),
		\\
		\frac{1}{4} M'+\frac{3}{4} f  & \subseteq \conv(\{x\} \cup M').
	\end{align*}
	
	This is equivalent to showing the following inclusions obtained by translating the right and the left hand sides by $-x$: 
	\begin{align*}
		\frac{1}{4} (M-x)+\frac{3}{4} (f-x)  & \subseteq \conv(\{0\} \cup (M'-x)), \\ \frac{1}{4} (M'-x)+\frac{3}{4} (f-x)  & \subseteq \conv(\{0\} \cup (M'-x)),
	\end{align*}
	where
	\begin{align*}
		M'-x & = (1+ \alpha) (M-x) + \alpha x + p,
		\\ f -x & = \mu (\alpha x + p).
	\end{align*}
	This shows that for proving the two inclusions we can assume $x=0$ (this corresponds to substitution of $M$ for $M-x$ and $p$ for $\alpha x + p$). So we assume $x=0$ and need to verify 
	\begin{align*}
		\frac{1}{4} M+\frac{3}{4} f  & \subseteq \conv(\{0\} \cup M'),
		\\ \frac{1}{4} M'+\frac{3}{4} f  & \subseteq \conv(\{0\} \cup M'),
	\end{align*}
	where $f = \mu p$ and $M' = (1+\alpha) M + p$. Since $\conv(\{0\} \cup M') = \bigcup_{0 \le \lambda \le 1} \lambda M'$, it suffices to verify to show that the left hand sides are subsets of $\lambda M'$ for some appropriate choices of $\lambda$. For the first inclusion we choose $\lambda = \frac{3}{4} \mu$. We observe that
	\[
		\frac{1}{4} M+\frac{3}{4} f = \frac{1}{4} M +\frac{3}{4} \mu p \subseteq \frac{3}{4} \mu (1+\alpha) M + \frac{3}{4} \mu p = \frac34 \mu ((1+\alpha) M + p) = \lambda M',
	\]
	where the second inclusion is fulfilled since $0\in M$ and $\mu\ge \frac{1}{3}$.
	
	For the second inclusion we choose $\lambda = \frac{3}{4} \mu + \frac{1}{4}$, which is at most 1 since $\mu \leq 1$. We observe that
	\[
		\frac{1}{4} M'+\frac{3}{4} f= \frac{1}{4} (1+\alpha) M + \frac{1}{4} p+\frac{3}{4} \mu p  \subseteq \left(\frac{3}{4} \mu + \frac{1}{4}\right) \big((1+ \alpha) M + p\big),
	\]
	where we use the fact that $\frac14 \leq \frac34 \mu + \frac14$, since $\mu \geq 0$.
\end{enumerate}
\end{proof}


 \subsection{On the approximability of $i$-hedral closures}\label{sec:proving_finiteness}

\begin{lemma}
	\label{lemma:polyhedron_relax}
	Let $M\subseteq \R^n$ be a lattice-free polyhedron given by $M = \{x\in \R^n:a_i\cdot x\leq b_i~\forall i\in [m]\}$ with $m\in \N, a_1, \dots, a_m\in \R^n\setminus \{0\}$, and $b_1, \dots, b_m\in \R$. Then there exists a nonempty subset $I$ of $[m]$ such that the polyhedron $P = \{x\in \R^n:a_i\cdot x\leq b_i~\forall i\in I\}$ can be represented as $P = \Delta_{n-k}\oplus U$, where $U$ is a $k$-dimensional linear subspace of $\R^n$ with $k \in \{0,\ldots,n-1\}$ and $\Delta_{n-k}$ is a $(n-k)$-dimensional simplex.
\end{lemma}

\begin{proof} 
Observe that $0\in \conv\{a_i:i\in [m]\}$. To see this, assume to the contrary that $0\not\in \conv\{a_i:i\in [m]\}$. By applying a separating hyperplane theorem for $0$ and $\conv\{a_i:i\in [m]\}$, we see that there exists a vector $u\in \R^n$ such that $a_i\cdot u<0$ for all $i\in [m]$. Since the recession cone of $M$ is equal to $\rec(M)=\{x\in \R^n:a_i\cdot x\leq 0~\forall i\in [m]\}$, we have that $u\in \intr(\rec(M))$. Hence $\rec(M)$ is $n$-dimensional. However, an implication of Theorem~\ref{thmLovasz}$(c)$ is that $\rec(M)$ cannot be $n$-dimensional because $M$ is lattice-free. Thus, we have a contradiction. 

Since $0\in  \conv\{a_i:i\in [m]\}$, Caratheodory's Theorem implies the existence of a nonempty subset $I\subseteq [m]$ such that the points $a_i$ with $i\in I$ are affinely independent and $0\in \conv\{a_i:i\in I\}$. We claim that $I$ is the desired subset. The set $V:=\aff(\conv \setcond{a_i}{i \in I})$ is a linear subspace of $\R^n$. We consider its orthogonal complement $U=V^\perp$. Note that $P = (P\cap V)\oplus U$. Since $0\in \conv\{a_i:i\in I\}$ and $\{a_i\}_{i\in I}$ is an affinely independent set, the set $U$ is a linear subspace of $\R^n$ of dimension $k = n - |I| + 1$, and the set $P\cap V$ is a $(n-k)$-dimensional simplex contained in the linear subspace $V$. 
\end{proof}

The following result is the main component in the proofs of Propositions~\ref{prop:finite_approx_f_arb} and~\ref{prop:proof_fixed_f_finite}. 

\begin{lemma}\label{prop:lifting}
	Let $L\in \cL_{\ast}^n$ and let $f\in \intr(L)$. Let $\gamma\in (0,1]$ and define $L':=\gamma (L-f)+f$. Assume that there are values $m\in \N$ and $t\in \Z$, and a maximal lattice-free set $D\in \cL_{m}^{n-1}$ such that
\begin{enumerate}[(a)]
\item $L\cap(\R^{n-1}\times\{t\})\subseteq D\times \{t\}$,
\item $w(L', e_n)\leq 1$, and
\item $L'\cap (\R^{n-1}\times \{t\})$ is nonempty.
\end{enumerate}
Then there exists a lattice-free set $B\in \cL_{m+1}^n$ such that $\frac{1}{4}\gamma(L-f)+f\subseteq B$. 
\end{lemma}

\begin{proof}
We introduce the following homothetical copy of $L$:
\begin{align*}
L''&:=\frac{1}{4}(L'-f)+f = \frac{1}{4}\gamma(L-f)+f.
\end{align*}
For $w\in \Z$, let $U_w:=\R^{n-1}\times \{w\}$. Without loss of generality, we assume that $t = 0$. Assumptions $(b)$ and $(c)$ imply that $L'\subseteq \conv(U_{-1}\cup U_1)$. If $\intr(L'')\cap U_0 = \emptyset$, then $L''$ is contained in $\conv(U_{b}\cup U_{b+1})$ for some $b\in \{-1,0\}$. In this case, the lattice-free split $B:= \conv(U_{b}\cup U_{b+1})$ gives the desired result. It is left to consider the case when $\intr(L'')\cap U_0\neq \emptyset$. To complete the proof, it is enough to identify a lattice-free set $B\in \cL_{m+1}^n$ such that $L''\subseteq B$.

Assume $\intr(L'')\cap U_0\neq \emptyset$. By assumption $(a)$, there exists an $(n-1)$-dimensional maximal lattice-free set $D \in \cL^{n-1}_m$ such that $D$ has $m$ facets and $L \cap U_0\subseteq D\times\{0\}$. The set $U_0 \setminus \relint(D\times\{0\})$ is a union of $m$ closed half-spaces $Q_1,\ldots,Q_m$ of the space $U_0$. Since $\gamma\in (0,1]$, it follows that $\intr(L')\cap U_0\subseteq \intr(L)\cap U_0\subseteq \relint(D\times\{0\})$. Thus, each of $Q_1,\ldots,Q_m$ is disjoint with $\intr(L')$. We apply a separation theorem for $Q_j$ and $\intr(L')$ to deduce the existence of a closed half-space $H_j$ of the space $\R^n$ with $Q_j\subseteq \R^n\setminus \intr(H_j)$ and $\intr(L')\subseteq H_j $. Thus, the polyhedron 
\begin{equation*}
	P:= H_1 \cap \ldots \cap H_m
\end{equation*}
contains $\intr(L')$, and therefore $L'\subseteq P$. 

Consider the $n$-dimensional polyhedron 
\begin{equation}\label{eqn:defn_B'}
	B':=P \cap \conv(U_{-1} \cup U_1) \in \cL_{m+2}^n
\end{equation}
in $\R^n$. We claim $B'$ is a lattice-free set in $\R^n$. By the definition of each $Q_i$ for $i\in [m]$ and using the containment  $Q_j\subseteq \R^n\setminus \intr(H_j)$, it follows that 
$$U_0\setminus \relint(D\times \{0\}) ~=~ \bigcup_{i=1}^m Q_i~\subseteq ~\R^n\setminus \bigcap_{i=1}^m \intr(H_i) = \R^n\setminus \intr(P).$$
Hence $\intr(P)\cap U_0\subseteq \relint(D\times\{0\})$ and $P\cap U_0\subseteq D\times\{0\}$. Since $D$ is a lattice-free set in $\R^{n-1}$, there are no lattice points in $\intr(P)\cap U_0$. Hence $B'$ is a lattice-free set. Also, since $L'\subseteq P$ and $L'\subseteq \conv(U_{-1}\cup U_1)$, it follows that 
\begin{equation}\label{eq:L'_contained_in_B'}
L'\subseteq B'.
\end{equation}

If $U_{-1} \cap B'$ or $U_1 \cap B'$ is not a facet of $B'$, then $B'$ has at most $m + 1$ facets. In this case, $B' \in \cL_{m+1}^n$ and we have the desired inclusion $L''\subseteq L' \subseteq B$ for $B:= B'$. 

If both $U_{-1} \cap B'$ and $U_1 \cap B'$ are facets of $B'$, finding an appropriate $B$ requires more work. Since $L'' \cap U_0 \ne \emptyset$ and $w(L'', e_n)\leq \frac{1}{4}$, we have that the $n$-th component $f_n$ of $f$ is contained in $[-\frac{1}{4},\frac{1}{4}]$. Without loss of generality, we assume that one even has $0 \le f_n\le \frac{1}{4}$. By Lemma~\ref{lemma:polyhedron_relax} applied to the lattice-free polyhedron $P\cap U_0$ in the space $U_0$, there exists a subset $J \subseteq [m]$ with $|J| \le n$ such that for the polyhedron
\[
	P_J := \bigcap_{j \in J} H_j,
\]
one has the representation
\[
	P_J\cap U_0 = \Delta \oplus V,
\]
where $\Delta$ is a simplex with $1 \le \dim(\Delta) \le n-1$ and $V$ is a linear subspace of $U_0$ with $\dim(V) \le n-2$. Since $B'\cap U_{-1}$ and $B'\cap U_1$ are facets of $B'$, the sets $P_J \cap U_{-1}$ and $P_J \cap U_1$ are $(n-1)$-dimensional. The representation $P_J\cap U_0= \Delta \oplus V$ implies that $P_J \cap U_{-1}$ and $P_J \cap U_1$ are homothetical copies of $P_J\cap U_0$. This implies that 
\begin{equation}\label{eq:defn_T}
	T := P_J\cap  \conv(U_{-1} \cup U_1)
\end{equation}
is a truncated cone with bases $P_J \cap U_{-1}$ and $P_J \cap U_1$. We claim that one can apply Lemma~\ref{lemma:tr_pyr}$(b)$ to $P_J\cap \conv(U_{-1} \cup U_1)$. 
Lemma~\ref{lemma:tr_pyr} can be used with either $M = P_J\cap U_1$ and $M' = P_J\cap U_{-1}$, or $M= P_J\cap U_{-1}$ and $M' = P_J\cap U_1$, depending on which of the two bases $P_J\cap U_1$ or $P_J\cap U_{-1}$ is larger. In both cases, we can verify the condition $\mu\geq \frac{1}{3}$. In the former case, we have $f = (1+\mu\alpha)x+\mu p$ for $\mu\in [0,1]$, $\alpha\geq 0$, $x\in P_J\cap U_1$, and $p\in \R^n$. Note that the $n$-th component of $p$ equals $p_n = -\alpha-2$ because $(1+\alpha)x+p\in P_J\cap U_{-1}$. Since $f_n\in[0, \frac{1}{4}]$, we can rearrange the equality $f_n = (1+\mu\alpha)+\mu p_n$ to obtain $\mu \geq \frac{3}{8}>\frac{1}{3}$. In the latter case, a similar argument can be used to show that $\mu \geq \frac{1}{2}>\frac{1}{3}$. 

Lemma~\ref{lemma:tr_pyr}$(b)$ yields the existence of a polyhedron $T'$ with $|J|+1$ facets satisfying $ \frac{1}{4} T + \frac{3}{4} f \subseteq T' \subseteq T$. Hence 
\[
	B := T'\cap \bigcap_{j \in [m] \setminus J} H_j \quad \subseteq \quad T  \bigcap_{j \in [m] \setminus J} H_j = B'
\]
where the last equality follows from~\eqref{eqn:defn_B'} and~\eqref{eq:defn_T}. Thus, $B$ is a lattice-free polyhedron since $B'$ is lattice-free. Moreover, $B$ has at most $m + 1$ facets. 
Since $f \in \intr(L') \subseteq P \subseteq \bigcap_{j \in [m] \setminus J} H_j$, and $B' = T  \bigcap_{j \in [m] \setminus J} H_j$ as noted above, we obtain $$\frac{1}{4} (B'-f) + f \subseteq\bigg(\frac{1}{4} (T-f) + f\bigg) \bigcap_{j \in [m] \setminus J} H_j\subseteq T'\cap \bigcap_{j \in [m] \setminus J} H_j = B.$$ From~\eqref{eq:L'_contained_in_B'}, we obtain the desired inclusion $$L'' = \frac{1}{4}(L'-f)+f\subseteq \frac{1}{4} (B'-f) + f \subseteq B.$$
~
\end{proof}

\begin{proposition}
	\label{prop:finite_approx_f_arb}
Let $i\in \N$ be such that $i>2^{n-1}$. Then $\rho(\cL_i^n,\cL_{*}^n)\leq 4\Flt(n)$.
\end{proposition}
\begin{proof}
It suffices to consider the case $i=2^{n-1} +1$, as every $\cL_i^n$ with $i > 2^{n-1}$ contains $\cL_{2^{n-1} + 1}^n$ as a subset. The assertion is trivial for $n=1$, so we assume that $n \ge 2$. By the definitions of the functionals $\rho_f(B,L)$ and $\rho_f(\cB,L)$, and with~\eqref{eqFamilySetvaryingf} and~\eqref{eqFamilyFamilyvaryingf}, it follows that
$$ \rho(\cL_i^n,\cL_{*}^n)\leq {\adjustlimits \sup_{L \in \cL_*^n, f \in \intr(L)} \inf_{B \in \cL_i^n}} \rho_f(B,L).$$

Let $L\in \cL_*^n$ and $f\in \intr(L)$. From the previous inequality, it is enough to show that there exists a $B\in \cL_{i}^n$ such that $\rho_f(B,L)\leq 4\Flt(n)$. By Proposition~\ref{propGeomRelativeStrength}, this condition is equivalent to the condition $\frac{1}{4}(L'-f)+f\subseteq B$, where 
$$
L' := \frac{1}{\Flt(n)} (L-f) +f.
$$
Thus, we aim to find a $B\in \cL_{i}^n$ such that $\frac{1}{4}(L'-f)+f\subseteq B$.

By Theorem~\ref{thm:flat}, there exists a $u\in \Z^n\setminus \{0\}$ such that $w(L,u)\leq \Flt(n)$. After a unimodular transformation, we assume $u = e_n$. For $t\in \Z$, let $U_t:=\R^{n-1}\times \{t\}$. If $L'\subseteq \intr(\conv(U_t\cup U_{t+1}))$ for some $t\in \Z$, then setting $B := \conv(U_t\cup U_{t+1})$ yields the desired result. Thus, we assume that $L'\cap (\R^{n-1}\times \{t\})$ is nonempty for some $t \in \Z$.

Since $L$ is lattice-free and $\intr(L)\cap U_t\neq \emptyset$, the set $\{x\in \R^{n-1} :(x,t)\in L\}$ is lattice-free. By Proposition~\ref{prop:all_in_max}, the latter set is a subset of a maximal lattice-free polyhedron $D$ in $\R^{n-1}$. By Theorem~\ref{thmLovasz}$(b)$, $D$ is in $\cL_{m}^{n-1}$ for $m= 2^{n-1}$. The desired conclusion now follows by applying Lemma~\ref{prop:lifting} with $L, L', D, m, t$ and $\gamma := \frac{1}{\Flt(n)}$.
\end{proof}

For $u \in \Q^n \setminus \{0\}$, recall that the denominator of $u$ is the minimal $s \in \N$ such that $s u \in \Z^n$. It is not hard to see that the denominator of a rational vector is invariant up to unimodular transformations.

\begin{proposition}
	\label{prop:proof_fixed_f_finite}
	Let $i \in \N$ be such that $i > n$. Let $f \in \Q^n \setminus \Z^n$, and let $s \in \N$ be the denominator of $f$. Then $\rho_f(\cL_i^n,\cL_\ast^n) \le \Flt(n) 4^{n-1} s$. 
\end{proposition}

\begin{proof}
	For each $L\in \cL_{*}^n\cap \cC_f^n$, we introduce two homothetical copies of $L$:
	\begin{align*}
	 L' & := \frac{1}{\Flt(n)4^{n-2} s} (L-f) + f,\\
	 L'' & := \frac{1}{4}(L'-f)+f = \frac{1}{\Flt(n)4^{n-1} s} (L-f) + f.
	\end{align*}
Using the definitions of $\rho_f(B,L)$ and $\rho_f(\cB,L)$ along with~\eqref{eq:fixed_f_fam_fam}, it follows that
$$ \rho_f(\cL_i^n,\cL_{*}^n)\leq {\adjustlimits \sup_{L \in \cL_*^n} \inf_{B \in \cL_i^n}} \rho_f(B,L).$$

By the previous inequality and Proposition~\ref{propGeomRelativeStrength}, it is enough to show that for an arbitrary $L \in \cL_{*}^n$ such that $f\in \intr(L)$, there exists a $B\in \cL_{i}^n$ such that $L''\subseteq B$. We verify that such a $B$ exists by induction on $n$. The assertion is clear for $n=1$ by setting $B=L$. Consider $n \ge 2$ such that for every $f' \in \Q^{n-1} / \Z^{n-1}$ with denominator $s$ and for every $\bar L \in \cL_{*}^{n-1}$, there exists $\bar B \in \cL_n^{n-1}$ satisfying $\bar L''\subseteq \bar B$. 

	 Let $L \in \cL_{*}^n$ with $f\in \intr(L)$. Let $u \in \Z^n \setminus \{0\}$ be a vector for which the lattice width of $L$ is attained. One has $u\cdot f \in \frac{1}{s} \Z$. After a unimodular transformation, we assume $u = e_n$  (recall that unimodular transformations do not change the denominator of rational vectors). For $w\in \Z$, define $U_w:=\R^{n-1}\times \{w\}$. Since $w(L', e_n)\leq 1$, there is some $m\in \Z$ such that $L'\subseteq\conv(U_{m-1}\cup U_{m+1})$. Without loss of generality, we assume that $m=0$, and so $L'\subseteq\conv(U_{-1}\cup U_{1})$. 
	
	Consider cases on the integrality of $f_n$. First, suppose $f_n\not\in \Z$. Without loss of generality, we assume that $f_n\in (0,1)$. Thus, $f\in \R^{n-1}\times[\frac{1}{s}, 1-\frac{1}{s}]$. Furthermore, $w(L'', e_n)\leq \frac{1}{s}$ by the choice of $L''$. Consequently, $L''$ is a subset of the lattice-free split $B:= \R^{n-1}\times[0,1]$.
	
	In the case $f_n\in \Z$, we use the induction assumption. Without loss of generality, we assume $f_n =0$. Thus, we can represent $f$ as $f= (f', 0)$, where $f' \in \Q^{n-1} \setminus \Z^{n-1}$ has the same denominator as $f$. Observe that $\setcond{x \in \R^{n-1}}{ (x,0) \in L}$ is a lattice-free polyhedron in $\R^{n-1}$. By Proposition~\ref{prop:all_in_max} and Theorem~\ref{thmLovasz}, there is an $(n-1)$-dimensional maximal lattice-free set $M_0$ such that 
	\begin{equation}\label{eq:fixed_f_contain_L_0}
	\setcond{x \in \R^{n-1}}{ (x,0) \in L}\subseteq M_0\in \cL_{2^{n-1}}^{n-1}.
	\end{equation}

	Applying the induction assumption to $M_0$, we obtain a lattice-free set $D\in \cL_n^{n-1}$ satisfying 
	\begin{equation*}\label{eq:make_D}
	\frac{1}{\Flt(n-1)4^{n-2} s} (M_0-f') + f'\subseteq D.
	\end{equation*}
	 By~\eqref{eq:fixed_f_contain_L_0} and the fact that $\Flt(n) \ge \Flt(n-1)$, we also have 
	\begin{equation*}\label{eq:fixed_f_shrink_L_0}
	L'\cap U_0\subseteq \left(\frac{1}{\Flt(n)4^{n-2} s} (M_0-f') + f'\right)\times\{0\}\subseteq D\times\{0\}.
	\end{equation*}
	
	Since $f_n=0$, the set $L'\cap U_0$ is nonempty. We use Lemma~\ref{prop:lifting} by taking $L$ in Lemma~\ref{prop:lifting} to be equal to $L'$ in this proof and choosing $m=n$, $\gamma=1$ and $t=0$. Thus, there is some $B\in \cL_{n+1}^n$ such that 
	$L'' = \frac{1}{4}L'+\frac{3}{4}f \subseteq B,$
	as desired. 
\end{proof}

\section{On the inapproximability of $i$-hedral closures}\label{sec:infinite_approx}

We use Theorem~\ref{thmFinitenessFamilyFamily}$(b)$ to show $\rho(\cL_i^n,\cL_{i+1}^n)=\infty$ for $2\leq i\leq 2^{n-1}$. We will show that for every $\mu\in (0,1)$, there exists an $L\in \cL_{i+1}^n$ and an $f \in \intr(L)$ such that $\mu (L-f) + f$ is not a subset of any $B \in \cL_i^n$. To this end, we start with the following lemma. 

\begin{lemma}
	\label{eq:shrink}
	Let $B\subseteq \R^n$ be a maximal lattice-free polyhedron with $m$ facets and let $c\in \intr(B)$. Let $z_1, \dots, z_m$ be integer points in the relative interior of the $m$ distinct facets of $B$. There exists an $\epsilon \in (0,1)$ such that, for all $i, j \in [m]$ with $i \ne j$, the interval $[z_i,z_j]$ has a nonempty intersection with $(1-\epsilon) B + \epsilon c$.
\end{lemma}

\begin{proof}
Consider the set $Z:= \{(z_i+z_j)/2:i,j\in [m], i\neq j\}$ of the midpoints of the segments appearing in the assertion. Note that $Z$ is a subset of $\intr(B)$. Clearly, $(1-\epsilon)\intr(B)+\epsilon c\supseteq Z$ holds when $\epsilon \in (0,1)$ is sufficiently small.
\end{proof}

\begin{proposition}
	\label{prop:infinite_approx_f}
Let $i\in \N$ be such that $i\leq 2^{n-1}$. Then $\rho(\cL_i^n,\cL_{i+1}^n)=\infty$.
\end{proposition}

\begin{proof}
	Let $n \ge 2$, as otherwise the assertion is trivial.
	By Proposition~\ref{prop:i_hedral_closed}, the family $\cL_i^n$ is $f$-closed for every $f\in \R^n\setminus \Z^n$. Thus, we can use Theorem~\ref{thmFinitenessFamilyFamily}$(b)$ for proving the assertion. Let $\mu \in (0,1)$. We will show that there exists a maximal lattice-free set $L'$ in $\R^n$ with exactly $i+1$ facets and a point $f \in \intr(L')$ such that every lattice-free polyhedron containing $L'_\mu:=\mu L'+ (1-\mu) f$ as a subset has at least $i+1$ facets. 
	
	By Lemmas~\ref{lemma:existence_lattice_free} and~\ref{eq:shrink}, there exists a maximal lattice-free subset $L$ of $\R^{n-1}$ with $i$ facets, a point $c \in \intr(L)$, $i$ integer points $z_1,\ldots,z_i\in \Z^{n-1}$ in the relative interior of the $i$ facets of $L$, and an $\epsilon \in (0,1)$ such that each segment $[z_j,z_k]$ with $j, k \in [i]$ and $j \ne k$ has a nonempty intersection with the interior of the homothetical copy $L_{1-\epsilon} := (1-\epsilon) (L-c)  + c$ of $L$. 
	
	First, assume that $L$ is bounded. We fix $f:= (c, \epsilon \mu) \in \R^n$. In order to identify an appropriate $L'$, we first introduce an auxiliary polytope $P$ in $\R^n$. We fix $P$ to be the pyramid $P:= \conv (\{f\} \cup F)$
	with apex $f$ and base 
	\[
	F := \left(\frac{1}{\epsilon \mu} (L-c) +c \right) \times \{-1\}.
	\]
	The cross-section of $P$ by the horizontal hyperplane $\R^{n-1}\times\{0\}$ is a homothetical copy of $L \times \{0\}$:
	\begin{equation}\label{eq:P_0_slice}
	P \cap (\R^{n-1} \times \{0\}) = \frac{1}{\epsilon \mu +1} f + \frac{\epsilon \mu}{ \epsilon \mu + 1} F = \left( \frac{1}{\epsilon \mu +1} (L-c) + c\right) \times \{0\}.
	\end{equation}
	Thus, $P \cap (\R^{n-1} \times \{0\}) \subseteq L \times \{0\}$ and, since $L$ is lattice-free, $P$ is also lattice-free. On the other hand, we have the inclusion $P \cap (\R^{n-1} \times \{0\}) \supseteq L_{1-\epsilon} \times \{0\}$, since $1- \epsilon \le 1 - \epsilon \mu \le \frac{1}{\epsilon \mu + 1}$. 
	
	We will now fix $L'$ to be a specific homothetical copy of $P$ with $f \in \intr(L')$ and $P \subseteq L'$, as follows. With each $\lambda \ge 1$ we associate the homothetical copy $P_\lambda := \lambda P + (1-\lambda) (c,-1)$ of $P$ with center at $(c,-1)$. We consider the cross-section of $P_{\lambda}$ by the hyperplane $\R^{n-1}\times \{0\}$:
	\begin{align*}
	P_{\lambda} \cap (\R^{n-1} \times \{0\}) =& (\lambda P) \cap (\R^{n-1}\times \{1-\lambda\}) +  (1-\lambda) (c,-1)\\
	=& \frac{1}{\lambda(\epsilon \mu +1)} f + \left(1- \frac{1}{\lambda(\epsilon \mu +1)}\right)F +(1-\lambda)(c, -1)\\
	=& \left( \frac{\lambda(\epsilon\mu+1)-1}{\epsilon\mu(\epsilon \mu +1)} (L-c) + c\right) \times \{0\}.
	\end{align*}
	We fix $\lambda = \frac{\epsilon\mu(\epsilon \mu +1)+1}{\epsilon\mu+1}$. With this choice, the cross-section $P_\lambda \cap (\R^{n-1} \times \{0\})$ coincides with $L \times \{0\}$. We fix $L' := P_\lambda$. By construction, $L'$ is a lattice-free pyramid with $i+1$ facets, $f \in \intr(L')$, and $P \subseteq L'$. See Figure~\ref{fig:example_of_L_and_P} for an example of $P$ and $L'$ in dimension $n=2$.
	
	According to Theorem~\ref{thmFinitenessFamilyFamily}$(b)$, for proving our assertion, it suffices to verify that every lattice-free polyhedron $M$  with $M \supseteq \mu (L' - f) + f$ has at least $i+1$ facets. Below we will verify an even stronger property that every lattice-free polyhedron $M$ with $\intr(M) \supseteq \mu \eps (P - f) + f$ has at least $i+1$ facets. We consider an arbitrary lattice-free polyhedron $M$ in $\R^n$ such that $\intr(M) \supseteq \mu \eps (P - f) + f$.
	
	Note that $\epsilon\mu (P-f)+f$ is a pyramid with apex $f$ and the base
	\[
	\epsilon \mu (F-f)+f =  L \times \{ - \epsilon^2 \mu^2 \}.
	\]
	Taking into account that the base of $\epsilon \mu (P-f) + f$ is below $\R^n \times \{0\}$, i.e., $-\epsilon^2\mu^2<0$, we conclude that the cross-section of $\epsilon \mu (P-f) + f$ by $\R^{n-1} \times \{0\}$ coincides with the cross-section of $P$ by $\R^{n-1} \times \{0\}$.  The latter implies $\epsilon \mu (P-f) + f \supseteq L_{1-\eps} \times \{0\}$.
	
	The set $\R^n \setminus \intr(M)$ can be represented as a union of closed half-spaces $H_1,\ldots,H_t$, where $t \in \N$. We need to show that $t \ge i+1$. Among the $i$ points $(z_1,0),\ldots,(z_i,0)$, no two distinct points fall into the same half-space $H_k$ with $k \in [t]$. If, say, $(z_1,0)$ and $(z_2,0)$ are both in $H_1$, then the segment $[(z_1,0),(z_2,0)] = [z_1,z_2] \times \{0\}$ is a subset of $H_1$. Since the latter segment has a nonempty intersection with $L_{1-\epsilon} \times \{0\}$, we get a contradiction to $L_{1-\eps} \times \{0\} \subseteq \epsilon\mu (P-f) +f\subseteq \intr(M)$. Thus, $t \ge i$, and without loss of generality, we assume that $(z_k,0) \in H_k$ for $k \in [i]$. 
	
	If $t<i+1$, we must have $t=i$. In this case, the point $(z_1,-1) \in \Z^n$ must belong to one of the hyperplanes $H_k$ with $k \in [i]$ since $M$ is lattice-free. The latter yields a contradiction as follows. Since the points $(z_1,-1)$ and $(z_k,0)$ both belong to $H_k$, their convex combination $q:=\epsilon^2 \mu^2 (z_1,-1) + (1-\eps^2 \mu^2) (z_k,0)$ also belongs to $H_k$ (see Figure~\ref{fig:example_of_L_and_P}). Clearly, $q \in L \times \{-\eps^2 \mu^2\}$, that is, $q$ belongs to the base of the pyramid $\eps \mu (P-f) + f$. We have thus found a point belong to $\eps \mu (P-f) + f$ but not to $\intr(M)$, which contradicts the choice of $M$. 	

	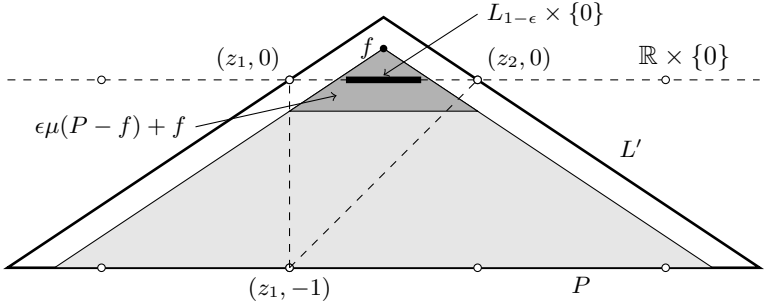
\begin{figure}[H]
	\centering
\begin{tikzpicture}[scale=2.5]
		\draw[line width=1pt] (-1.5,0) -- (2.5,0)  --  (1/2,4/3) node[pos = .4, above right]{\small$L'$} -- cycle;
		\filldraw[fill=black!10!white] (1/2-7/4,0) -- (1/2+7/4,0) node[below, pos = .8]{\small $P$}  -- (1/2, 7 / 6) -- cycle;
		\filldraw[fill=black!30!white] (0,10/12) -- (1,10/12)-- (1/2,7/6) -- cycle;
		\draw[->](-.5, .75)--(.25, 11/12)node[left, pos = 0]{\small $\epsilon\mu (P-f)+f$};
		\fill (1/2,7/6) circle (0.02) node[left]{\small$f$};
		\draw[dashed] (-1.5,1) -- (2.5,1) node[above, pos = .9]{$\R\times\{0\}$};
		\draw[line width=2.5pt] (0.3,1) -- (0.7,1);
		\draw[->](1, 1.35)--(.5, 1.025) node[right, pos = 0]{\small$L_{1-\epsilon}\times\{0\}$};
		\foreach \x in {-1,...,2} 
		\foreach \y in {0,...,1}
		{
			\filldraw[draw=black,fill=white] (\x,\y) circle (0.02);
		}
		\filldraw[draw=black,fill=white] (0,0) circle (0.02) node[below]{\small$(z_1, -1)$};
		\filldraw[draw=black,fill=white] (0,1) circle (0.02) node[above left]{\small$(z_1, 0)$};
		\filldraw[draw=black,fill=white] (1,1) circle (0.02) node[above right]{\small$(z_2, 0)$};
		\draw[dashed](0,0) --(0,1);
		\draw[dashed](0,0) --(1,1);
\end{tikzpicture}
\label{fig:example_of_L_and_P}\caption{Illustration (for $n=2$) of the proof of Lemma~\ref{prop:infinite_approx_f}. The darker shaded set $\epsilon\mu (P-f)+f$ and the lighter shaded $P$ are homothetical copies of $L'$ contained within $L'$. The thick black segment is $L_{1-\epsilon}\times\{0\}$. The points $(z_1, 0), (z_2,0),$ and $(z_1, -1)$ are contained in distinct facets of $L'$, and any half-space containing two points in $\{(z_1, 0), (z_2,0), (z_1, -1)\}$ necessarily contains points in $\epsilon\mu (P-f)+f$, which is contained in $\mu P+(1-\mu)f$.}
\end{figure}

	It is left to consider the case when $L$ is unbounded. Theorem~\ref{thmLovasz}$(c)$, states that after an appropriate unimodular transformation, $L$ can be written as $L = M\times \R^{n-1-n'}$, where $n'\in \N$ is such that $1\leq n'\leq n-2$ and $M$ is a bounded, maximal lattice-free set in $\R^{n'}$ with $i$ facets. By Theorem~\ref{thmLovasz}$(b)$, $i\leq 2^{n'} = 2^{(n'+1)-1}$. Applying the bounded case in dimension $n'+1$, there exists a maximal lattice-free set $L''$ in $\R^{n'+1}$ and an $f''\in \R^{n'+1}$ such that any lattice-free polyhedron in $\R^{n'+1}$ containing $L''_\mu:=\mu L''+ (1-\mu) f''$ as a subset has at least $i+1$ facets. Define $L':=L''\times \R^{n-n'-1}$ and $f = (f'', 0)\in \R^{n'+1}\times \R^{n-n'-1}$. Let $B$ be a lattice-free polyhedron in $\R^n$ containing $\mu L'+(1-\mu)f$. Then $B'' := \{x\in \R^{n'+1}: (x,0)\in B\}$ is a lattice-free polyhedron in $\R^{n'+1}$ containing  $\mu L''+(1-\mu)f''$. Hence $B''$ must have $i+1$ facets, and as a consequence, $B$ must also have $i+1$ facets. This gives the desired result. \end{proof}

\begin{proof}[Proof of Theorem~\ref{thmRelativeStrength}]
If $i>2^{n-1}$, then by Proposition~\ref{prop:finite_approx_f_arb} it follows that $\rho(\cL_i^n, \cL_*^n)$ $\leq 4\Flt(n)$. If $i\leq 2^{n-1}$, then by Proposition~\ref{prop:infinite_approx_f} it follows that $\rho(\cL_i^n, \cL_{i+1}^n) =\infty $.
\end{proof}

Similarly to Proposition~\ref{prop:infinite_approx_f}, proving $\rho_f(\cL_i^n, \cL_*^n)=\infty$ requires us to identify, for every choice of $\mu\in (0,1)$, some polyhedron $L\in \cL_k^n$ with $k>i$ satisfying $B\not\supseteq \mu L+(1-\mu)f$ for every $B\in \cL_i^n$. However, unlike in the proof of Proposition~\ref{prop:infinite_approx_f} where the choices of $L$ and $f$ are allowed to depend on $\mu$, proving $\rho_f(\cL_i^n, \cL_*^n)=\infty$ requires us to handle the situation where $f$ is fixed a priori, independent of $\mu$. The next result takes care of this.

\begin{lemma} \label{lem:non:approx:fixed:f}
	Let $f \in \Q^n \setminus \Z^n$ and $\alpha > 1$. Then the following hold:
	\begin{enumerate}[(a)]
		\item There exists a maximal lattice-free simplex $L \in \cL_{n+1}^n\cap \cC_f^n$ such that, for some choice of $n+1$ integer points $z_1,\ldots,z_{n+1}$ in the relative interior of the $n+1$ distinct facets of $L$, the following is fulfilled: for all $i,j \in [n+1]$ with $i \ne j$, the segment $[z_i,z_j]$ intersects the homothetical copy $L_\alpha := \frac{1}{\alpha} L + (1-\frac{1}{\alpha}) f$ of the simplex $L$. 
		\item For every simplex $L$ as in $(a)$, one has $\rho_f(B,L) \ge \alpha$ for every $B \in \cL_n^{n}$.
	\end{enumerate}
\end{lemma}

\begin{proof}

	$(a)$: We argue by induction on $n$. In the case $n=1$, we can choose $L=[z_1,z_2]$, $z_1 := \lfloor{f\rfloor}$ and $z_2 :=  \lceil{f\rceil}$. Now assuming the assertion is true for some dimension $n \in \N$, we verify the assertion is true for dimension $n+1$. Let $f \in \Q^{n+1} \setminus \Z^{n+1} $ and $\alpha > 1$. Applying a linear unimodular transformation, we assume that the last component of $f$ is zero so that $f$ has the form $f=(f',0)$ with $f' \in \Q^n \setminus \Z^n$. By the induction assumption, there exists a maximal lattice-free simplex $L' \in \cL_{n+1}^n$ such that, for some choice of $n+1$ integer points $z_1',\ldots,z_{n+1}'$ in the relative interior of the $n+1$ distinct facets of $L'$ and for all $i,j \in [n+1]$ with $i \ne j$, the segment $[z_i',z_j']$ intersects $L'_\alpha := \frac{1}{\alpha} L + (1-\frac{1}{\alpha})f$. 
	
	We choose $L \in \cL_{n+2}^{n+1}$ to be the pyramid
	\(
	L:=\conv(\{a\} \cup F)
	\)
	with apex
	\[
	a:=\left(f', \frac{1}{\alpha-1}\right)
	\] 
	and base 
	\[
	F:=(\alpha L' + (1-\alpha) f') \times \{-1\}.
	\]
	
	The cross-section $L \cap (\R^n \times \{0\})$ of $L$ can be given as $(1-\frac{1}{\alpha}) a + \frac{1}{\alpha} F$, which is seen to coincide with $L' \times \{0\}$. The apex $a$ lies above the points $(f',0) \in L'\times \{0\}$ and $(f',-1) \in F$. Fix $z_i:=(z_i',0)$ for every $i \in [n+1]$, and fix $z_{n+2}$ to be any point of the form $(z_i',-1)$ with $i \in [n+1]$, say $z_{n+2} := (z_1',-1)$. The assumption $\alpha > 1$ implies $L' \times \{-1\} \subseteq \relint(F)$. Thus, $z_{n+2} \in \relint(F)$, and the chosen $L$ is indeed a maximal lattice-free simplex with the integer points $z_1,\ldots,z_{n+2}$ in the relative interior of the $n+2$ distinct facets of $L$. We also observe that the base of the homothetical copy $L_\alpha$ of $L$ has the form
	\begin{equation} \label{eq:F:gamma}
	F_\alpha =\frac{1}{\alpha} F + \left(1- \frac{1}{\alpha}\right) f 
		= L' \times \left\{ - \frac{1}{\alpha} \right\}.
	\end{equation}
	That is, the base $F_\alpha$ of $L_\alpha$ and the cross-section $L \cap (\R^n \times \{0\})$ of $L$ are copies of $L'$ that lie in $\R^{n+1}$ one above the other. See Figure~\ref{fig:example_of_L} for an example of $f$, $L$, and $L_{\alpha}$ in dimension $n+1=2$. 	
	
	Consider $i, j \in [n+2]$ with $i \ne j$. If $i$ or $j$ is $n+2$, say $j=n+2$, then it can be checked that the point $\left(1-\frac{1}{\alpha}\right) z_i + \frac{1}{\alpha} z_{n+2}$ belongs to the facet $F_\alpha$ of $L_\alpha$. If neither $i$ nor $j$ coincides with $n+2$, the induction hypothesis yields that the segment $[z_i',z_j']$ intersects  $L'_\alpha$. This implies the segment $[z_i,z_j]$ intersects $L_\alpha$.
	\begin{figure}[H]
	\centering
\begin{tikzpicture}[scale=2.5]
	\draw[very thick] (1/2,1/4) node[above]{$a$}-- (3,-1) -- node[pos = .3, below right]{$L$}  (-2,-1) -- cycle;
	\filldraw[black!15!white,draw=black,thick] (0,-1/5) -- node[black, pos = .8, below right]{$L_{\alpha}$} (1,-1/5) -- (1/2,1/20) -- cycle;	
	\draw (0,0) -- (1,0);
	\draw[dashed,thick] (1,0) -- (0,-1);
	\draw[dashed,thick]  (0,0) -- (0,-1);
	\foreach \x in {-2,...,3} {
		\filldraw[white,draw=black,thick] (\x,-1) circle (0.02);
	}
	\draw[dashed,thick] (1/2,1/4) -- (1/2,0);
	\draw[dashed](-2, 0)--(3,0) node[above, pos = .9]{$\R\times\{0\}$};
	\fill (1/2,0) node[below right]{$f$} circle (0.02);
	\filldraw[white,draw=black,thick] (0,0) circle (0.02) node[above left]{\color{black}$z_1$} (1,0) circle (0.02) node[above right]{\color{black}$z_2$}  (0,-1) circle (0.02) node[below left]{\color{black}$z_3$} ;
\end{tikzpicture}
\label{fig:example_of_L}\caption{Illustration of the inductive step (for $n+1=2$) in the proof of Lemma~\ref{lem:non:approx:fixed:f}. The shaded simplex $L_{\alpha}$ is a homothetical copy of the larger simplex $L$. The values $z_1, z_2, z_3$ are contained in the relative interior of the distinct facets of $L$, and for every pair $i,j\in [3]$ with $i\neq j$, the segment $[z_i, z_j]$ intersects $L_{\alpha}$.}
\end{figure}
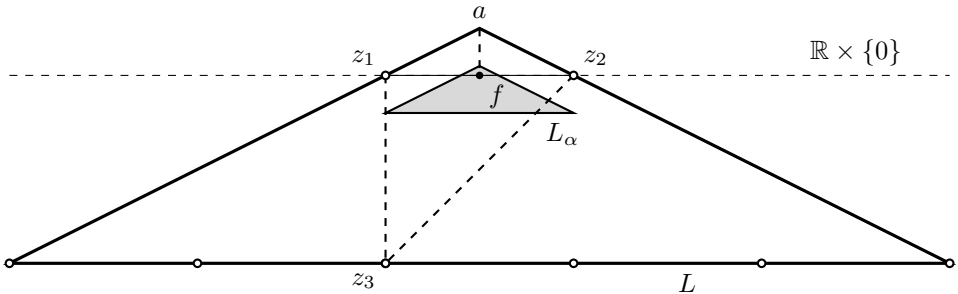

	$(b)$: Assume to the contrary that $\rho_f(B,L) < \alpha$ holds for some $B \in \cL_n^{n}$. By Proposition~\ref{propGeomRelativeStrength}, there is some $0 \le \lambda < \alpha$ such that $B \supseteq L_{\lambda} :=\frac{1}{\lambda} L + (1- \frac{1}{\lambda} )f$. Thus, $L_\alpha \subseteq \intr(L_\lambda) \subseteq \intr(B)$. Since $B \in \cL_n^{n}$, the interior of $B$ is the intersection of $n$ open half-spaces that have no common points in $\Z^n$. Consequently, $\R^n \setminus \intr(B)$ is a union of closed half-spaces $H_1, \ldots, H_n$, which cover $\Z^n$. Since the number of half-spaces $H_1,\ldots,H_n$ is smaller than the number of points $z_1,\ldots,z_{n+1}$, there is a half-space $H_j$ that covers at least two of the points $z_1,\ldots,z_{n+1}$. But then such a $H_j$ also covers the segment joining these two points. By $(a)$, such a segment intersects $L_\alpha$. We have thus shown that $L_\alpha \setminus \intr(B) \ne \emptyset$, which is a contradiction.
\end{proof}

\begin{proposition}
	\label{prop:proof_fixed_f_infinite}
	Let $i\in \N$ be such that $i\leq n$. Let $f\in \Q^n\setminus \Z^n$. Then $\rho_f(\cL_i^n,\cL_{i+1}^n ) = \infty$. 
\end{proposition}

\begin{proof}

	Let $\mu\in (0,1)$. By Theorem~\ref{thmFinitenessFamilyFamily}$(a)$ and Proposition~\ref{prop:i_hedral_closed}, it suffices to show that there exists $L \in \cL_{i+1}^n\cap \cC_f^n$ satisfying $B\not\supseteq \mu L+(1-\mu)f$ for all $B\in \cL_i^n\cap \cC_f^n$. To this end, let $\lambda\in (0,\mu)$. For $i=n$, Lemma~\ref{lem:non:approx:fixed:f} implies there is some $L \in \cL_{i+1}^n\cap \cC_f^n$ such that $\rho_f(B,L)\geq \frac{1}{\lambda}>\frac{1}{\mu}$ for all $B\in \cL_n^{n}$. By Proposition~\ref{propGeomRelativeStrength}, this implies that $B\not\supseteq \mu L+(1-\mu)f$ for all $B\in \cL_i^n\cap \cC_f^n$, as desired. 
	
	Consider the case $i < n$. After applying an appropriate unimodular transformation, we assume that $f = (f',0,\ldots,0) \in \R^n$ for some $f'\in \Q^{i} \setminus \Z^{i}$. An application of Lemma~\ref{lem:non:approx:fixed:f} in dimension $i$ yields the existence of a  maximal lattice-free simplex $L' \in \cL_{i+1}^i$ such that $\rho_f(B',L')\geq \frac{1}{\lambda}>\frac{1}{\mu}$ for all $B'\in \cL_i^{i}$. We choose $L=L' \times \R^{n-i}$ and show that $B\not\supseteq \mu L+(1-\mu)f$ for every $B \in \cL_i^n$. The homothetical copy $\mu L + (1-\mu) f$ of $L$ contains the affine space $A:=\{f'\} \times \R^{n-i}$. If $B\not\supseteq A$, we get $B\not\supseteq \mu L+(1-\mu)f$. Otherwise, $B \supseteq A$  and thus $B$ can be represented as $B=B' \times \R^{n-i}$ with $B' \in \cL_i^i$. In this case, $B\not\supseteq \mu L+(1-\mu)f$ since $B'\not\supseteq \mu L'+(1-\mu)f'$, which holds because $\rho_f(B',L')>\frac{1}{\mu}$ for all $B'\in \cL_i^{i}$.
	\end{proof}

\begin{proof}[Proof of Theorem~\ref{thmRelativeStrength_fixed_f}]
If $i>n$, then by Proposition~\ref{prop:proof_fixed_f_finite} it follows that $\rho_f(\cL_i^n, \cL_*^n) \leq \Flt(n) 4^{n-1} s$. If $i\leq n$, then by Proposition~\ref{prop:proof_fixed_f_infinite} it follows that $\rho_f(\cL_i^n, \cL_*^n) = \infty$.
\end{proof}

\bibliographystyle{siamplain}
\bibliography{full-bib}

\end{document}